\documentclass[reqno]{amsart}
\usepackage{amsmath,amsfonts,amssymb,mathrsfs}
\usepackage{hyperref}

\usepackage{mathtools}
\usepackage[inner=1.0in,outer=1.0in,bottom=1.0in, top=1.0in]{geometry}

\usepackage[symbol]{footmisc}

\newtheorem{theorem}{Theorem}[section]

\newtheorem{proposition}[theorem]{Proposition}
\newtheorem{corollary}[theorem]{Corollary}
\newtheorem{lemma}[theorem]{Lemma}
\newtheorem{remark}[theorem]{Remark}

\newcommand{\CC}{\mathbb{C}}

\newcommand{\RR}{\mathbb{R}}

\newcommand{\HH}{\mathbb{H}}

\newif\iffinalrun
\iffinalrun
\else
 \fi

\iffinalrun
  \newcommand{\need}[1]{}
  \newcommand{\mar}[1]{}
\else
  \newcommand{\need}[1]{{\tiny *** #1}}
  \newcommand{\mar}[1]{\marginpar{\raggedright\tiny Fix Me:  #1 }}\fi

\numberwithin{equation}{section}

\begin{document}
\title{Intersecting geodesics on the modular surface}

\author{Junehyuk Jung and Naser Talebizadeh Sardari}

\address{Department of Mathematics, Brown University, Providence, RI 02912 USA}
\email{junehyuk{\textunderscore}jung@brown.edu}
\address{Penn State department of Mathematics, McAllister Building, Pollock Rd, State College, PA 16802 USA}
\email{nzt5208@psu.edu}

\thanks{We thank D. Jakobson, V. Blomer, D. Milicevic, C. Pagano, M. Lee, M. Lipnowski, Y. Kim, and J. Yim for valuable comments. J.J. thanks A. Reid for the discussion that led to this project. J.J. was supported by NSF grant DMS-1900993, and by Sloan Research Fellowship. N.T.S. was supported partially by the National Science Foundation under Grant No. DMS-2015305, and is grateful to Max Planck Institute for Mathematics in Bonn for its hospitality and financial support.}

\begin{abstract}
We introduce the \textit{modular intersection kernel}, and we use it to study how geodesics intersect on the full modular surface $\mathbb{X}=PSL_2\left(\mathbb{Z}\right) \backslash \mathbb{H}$.  Let $C_d$ be the union of closed geodesics with discriminant $d$ and let $\beta\subset \mathbb{X}$ be a compact geodesic segment. As an application of Duke's theorem to the modular intersection kernel, we prove that $ \{\left(p,\theta_p\right)~:~p\in \beta \cap  C_d\}$ becomes equidistributed with respect to $\sin \theta ds d\theta$ on $\beta \times [0,\pi]$ with a power saving rate as $d \to +\infty$. Here $\theta_p$ is the angle of intersection between $\beta$ and  $C_d$  at $p$. This settles the main conjectures introduced by Rickards \cite{rick}.

We prove a similar result for the distribution of angles of intersections between $C_{d_1}$ and $C_{d_2}$ with a power-saving rate in $d_1$ and $d_2$ as $d_1+d_2 \to \infty$. Previous works on the corresponding problem for compact surfaces do not apply to $\mathbb{X}$, because of the singular behavior of the modular intersection kernel near the cusp. We analyze the singular behavior of the modular intersection kernel by approximating it by general (not necessarily spherical) point-pair invariants on $PSL_2\left(\mathbb{Z}\right) \backslash PSL_2\left(\mathbb{R}\right)$ and then by studying their full spectral expansion.

\end{abstract}
\maketitle
\section{Introduction}
Let $Y$ be a negatively curved surface of finite area. The prime geodesic theorem \cite{PGT} states that the number of primitive closed geodesics having length less than $L$, which we denote by $\pi\left(L\right)$, satisfies
\[
\pi\left(L\right)\sim \frac{e^L}{L},
\]
as $L \to \infty$. A natural problem is to understand how primitive closed geodesics of length less than $L$ are positioned or distributed in $Y$ as $L \to \infty$. In particular, one may ask
\begin{enumerate}
\item how the number of transversal intersections $I\left(\alpha_1, \alpha_2\right)$ between two primitive closed geodesics $\alpha_1$ and $\alpha_2$ is distributed, or
\item how the set of angles of intersections between $\alpha_1$ and $\alpha_2$ is distributed,
\end{enumerate}
as one varies $\alpha_1$, or both $\alpha_1$ and $\alpha_2$. Bonahon \cite{MR847953} defined the \textit{intersection form} $i:\mathcal{C} \times \mathcal{C} \to \mathbb{R}^+$ on the space of currents $\mathcal{C}$ such that when $\mu_i$ ($i=1,2$) is the unique invariant measure corresponding to $\alpha_i$, then $i\left(\mu_1,\mu_2\right) = I\left(\alpha_1,\alpha_2\right)$. When $Y$ is compact, Pollicott and Sharp \cite{PS} used an extension of the intersection form to understand the distribution of angles of self-intersections of closed geodesic $\alpha$ having length less than $L$, as $L \to \infty$. When $Y$ is a compact hyperbolic surface, using the intersection form, Herrera \cite{herrera} proved that the distribution of $I\left(\alpha_1, \alpha_2\right)/\left(l\left(\alpha_1\right)l\left(\alpha_2\right)\right)$ for closed geodesics $\alpha_1,\alpha_2$ of length $<L$, is concentrated near $1/\left(2\pi^2 \left(g-1\right)\right) = 2/\left(\pi \mathrm{vol}\left(Y\right)\right)$ with exponentially decaying tail, as $L \to \infty$. Here $l\left(\cdot\right)$ is the length function, and $g$ is the genus of $Y$.

In this article, we study a refined problem:
\begin{itemize}
\item[(3)] how are the locations and angles of intersections between $\alpha_1$ and $\alpha_2$ jointly distributed relative to $\alpha_2$, as one varies $\alpha_1$, or both $\alpha_1$ and $\alpha_2$?
\end{itemize}

To state our main theorem, we let $\mathbb{X}$ be the full modular surface $PSL_2\left(\mathbb{Z}\right)\backslash \mathbb{H}$. On $\mathbb{X}$, primitive oriented closed geodesics are in one-to-one correspondence with conjugacy classes of primitive hyperbolic elements of  $PSL_2\left(\mathbb{Z}\right)$. Moreover there is a bijection between the primitive  hyperbolic conjugacy classes and the $SL_2\left(\mathbb{Z}\right)$ equivalence classes  of primitive integral binary quadratic forms of non-square positive discriminant \cite{LRS,sar82}. So by the discriminant of a primitive closed geodesic, we mean the discriminant of the corresponding binary quadratic form. In particular, if the hyperbolic class $\gamma$ is associated to the binary quadratic form $Q$ then $\gamma^{-1}$ is associated to $-Q$.

Each oriented primitive closed geodesics of discriminant $d$ has a unique lift to a closed geodesic of length $2 \log \varepsilon_d$ in the unit tangent bundle $S\mathbb{X}$. Let $h\left(d\right)$ be the number of inequivalent primitive integral binary quadratic forms of discriminant $d$. We denote the disjoint union of these $h\left(d\right)$ closed geodesics by $\mathscr{C}_d \subset S\mathbb{X},$ which has total length $2h\left(d\right)\log \varepsilon_d$.

 Note that the closed geodesic on $\mathbb{X}$ has length $ \log \varepsilon_d$ or $2 \log \varepsilon_d$  according as $Q$ is or is not equivalent to $-Q$ \cite[p.75]{duke}.
  We now let $C_d$ be the union of primitive (unoriented) closed geodesics of discriminant $d$ on $\mathbb{X}$, and note that $l\left(C_d\right) = h\left(d\right) \log \varepsilon_d $ is the total length of $C_d$.

\begin{theorem}\label{main1}
Fix $T > 100$, and let $\beta$ be a compact oriented geodesic segment of length $<1$  in the region determined by $y<T$ on $\mathbb{X}$. For $0<\theta_1<\theta_2<\pi$, let $I_{\theta_1,\theta_2}\left(\beta,C_d \right)$ be the number of intersections between $\beta$ and $C_d$ with the angle between $\theta_1$ and $\theta_2$. (Here the angle between $\beta$ and $C_d$ at $p \in \beta \cap C_d$ is measured counterclockwise from the tangent to $\beta$ at $p$ to the tangent to $C_d$ at $p$.)

Then we have
\[
\frac{I_{\theta_1,\theta_2}\left(\beta,C_d\right)}{l\left(\beta\right)l\left(C_d\right)} = \frac{3}{\pi^2} \int_{\theta_1}^{\theta_2} \sin \theta d\theta   + O_\epsilon\left(d^{-\frac{25}{3584}+\epsilon}\right),
\]
uniformly in $\beta$, $\theta_1$, and $\theta_2$, under the assumption that
\[
\theta_2-\theta_1 \gg d^{-\frac{25}{7168}},\footnote[2]{Here and elsewhere, $A\ll_\tau B$ means $|A| \leq C(\tau) B$ for some constant $C(\tau)$ that depends only on $\tau$.}
\]
and that
\[
l\left(\beta\right) \gg d^{-\frac{25}{7168}}.
\]
\end{theorem}
\begin{remark}
This statement is false if $C_d$ is replaced by individual geodesics. For instance, the set of intersections between $\beta$ and a closed geodesic $\alpha$ does not necessarily become equidistributed as $l\left(\alpha\right) \to \infty$. To see this, take a finite sheeted covering $S$ of $\mathbb{X}$ whose genus is $\geq 2$. Then according to Rivin's work \cite{riv} there are arbitrarily long simple closed geodesics on $S$. Note that these simple closed geodesics must be contained in a compact part of $S$ \cite{JR}. This implies that there is a compact set $C \subset \mathbb{X}$ which contains arbitrarily long primitive closed geodesics. Take a geodesic segment $\beta$ in $\mathbb{X}-C$. Then there are infinitely many closed geodesics which do not intersect $\beta$.
\end{remark}
\begin{remark}
The exponent $-\frac{25}{3584}$ can be improved slightly by refining our argument, but in order to keep the exposition simple, we do not discuss the optimal rate in the current article.
\end{remark}
As an immediate consequence, we deduce that the intersection points and corresponding angles become equidistributed, resolving the main conjectures introduced by Rickards \cite{rick}. \begin{corollary}
Fix a closed geodesic $\alpha$. Then for any fixed segment $\beta \subset \alpha$, and any fixed $0<\theta_1<\theta_2<\pi$, we have
\[
\lim_{d\to \infty}\frac{I_{\theta_1,\theta_2}\left(\beta,C_d\right)}{I\left(\alpha,C_d\right)} = \frac{l\left(\beta\right)}{l\left(\alpha\right)}\int_{\theta_1}^{\theta_2} \frac{\sin \theta }{2}d\theta.
\]
\end{corollary}
\begin{remark}
Rickards's work is motivated by the work of Darmon and Vonk~\cite{DV2} on the arithmetic ($p$-arithmetic) intersection between pairs of oriented closed geodesics on the modular surfaces (Shimura curves). The arithmetic intersection between oriented closed geodesics $\alpha_1$ and $\alpha_2$ of discriminants $D_1$ and $D_2$ only depends on $D_1$ and $D_2$ and the angles of intersections between $\alpha_1$ and $\alpha_2$.  Darmon and Vonk   conjectured that \cite[Conjecture 2]{DV2} that the $p$-arithmetic intersection is algebraic and belongs to the composition of the Hilbert class field of real quadratic fields of discriminants $D_1$ and $D_2$.
\end{remark}
To prove our main results, we introduce the \textit{modular intersection kernel}. For $\delta>0$ and $\theta_1, \theta_2 \in \left(0,\pi\right)$, let $k_\delta^{\theta_1,\theta_2}:S\mathbb{H} \times S\mathbb{H} \to \mathbb{R}$ be the integral kernel defined by
\[
k_\delta^{\theta_1,\theta_2}\left(\left(x_1,\xi_1\right), \left(x_2,\xi_2\right)\right) =1,
\]
if the geodesic segments of length $\delta$ from $x_i$ with the initial vector $\xi_i$ intersect at an angle $\in \left(\theta_1,\theta_2\right)$, and $0$ otherwise. Under the identification $S\mathbb{H} \cong PSL_2\left(\mathbb{R}\right)$, for a given discrete subgroup $\Gamma \subset PSL_2\left(\mathbb{R}\right)$, we define the \textit{modular intersection kernel} $K_\delta^{\theta_1,\theta_2}: \Gamma \backslash PSL_2\left(\mathbb{R}\right)\times \Gamma \backslash PSL_2\left(\mathbb{R}\right) \to \mathbb{R}$ by taking the average of $k_\delta^{\theta_1,\theta_2}$ over $\Gamma$:
\[
K_\delta^{\theta_1,\theta_2} \left(g_1,g_2\right) = \sum_{\gamma \in \Gamma}k_\delta^{\theta_1,\theta_2} \left(g_1,\gamma g_2\right).
\]
The basic idea of the proof of Theorem \ref{main1} then is as follows. Heuristically,
\[
I_{\theta_1,\theta_2}\left(\beta,C_d\right)
\]
should be well approximated by
\begin{equation}\label{integral1}
\frac{1}{2\delta^2} \int_{\mathscr{C}_d} \int_{\widetilde{\beta}}  K_\delta^{\theta_1,\theta_2} \left(s_1,s_2\right)ds_1ds_2,
\end{equation}
where  $\widetilde{\beta}\subset S\mathbb{X}$ is a lift of $\beta$ with either of orientations of $\beta$
\[
\widetilde{\beta}\left(t\right) = \left(\beta\left(t\right),\beta'\left(t\right)\right),
\]
under assuming that $\beta\left(t\right)$ is parameterized by the arc length. As noted from \cite{LRS}, Duke's theorem \cite{duke} can be extended to the equidistribution of $\mathscr{C}_d$ in $PSL_2\left(\mathbb{Z}\right) \backslash PSL_2\left(\mathbb{R}\right)$ as $d \to \infty$. Observing that
\begin{equation}\label{integral2}
\frac{1}{2\delta^2} \int_{\widetilde{\beta}}  K_\delta^{\theta_1,\theta_2} \left(s_1,g\right)ds_1
\end{equation}
is a compactly supported function in $g$ for compact $\beta$, \eqref{integral1} is
\[
\sim \frac{l\left(\mathscr{C}_d\right)}{2\delta^2} \int_{g} \int_{\widetilde{\beta}}  K_\delta^{\theta_1,\theta_2} \left(s_1,g\right)ds_1d\mu_g,
\]
which is asymptotically $\frac{3}{\pi^2}l\left(C_d\right)l\left(\beta\right) \int_{\theta_1}^{\theta_2} \sin \alpha d\alpha$ as $\delta \to 0$, by an explicit computation.

Note that \eqref{integral2} is a discontinuous function. Therefore, in order to obtain the rate of convergence, we need a smooth approximation of \eqref{integral2}, and a quantified version of Duke's theorem with explicit dependency on the test functions. To this end, we follow the argument sketched in \cite{LRS} to prove:
\begin{theorem}\label{effect}
Assume that $f \in C_0^\infty \left(PSL_2\left(\mathbb{Z}\right) \backslash PSL_2\left(\mathbb{R}\right)\right)$ has support in the region determined by $y<T$. Then we have
\[
\frac{1}{l\left(\mathscr{C}_d\right)} \int_{\mathscr{C}_d} f\left(s\right) ds = \frac{3}{\pi^2}\int_{PSL_2\left(\mathbb{Z}\right) \backslash PSL_2\left(\mathbb{R}\right)} f\left(g\right) d\mu_g + O_\epsilon\left( \log T d^{-\frac{25}{512}+\epsilon}\|f\|_{W^{6,\infty}}\right).
\]
\end{theorem}
Here $\|\cdot\|_{W^{k,p}}$ is the Sobolev norm:
\[
\|f\|_{W^{k,p}} = \max_{|\alpha|\leq k}\|\partial_\theta^{\alpha_1} \left(y\partial_x\right)^{\alpha_2} \left(y\partial_y\right)^{\alpha_3} f\|_{L^p}.
\]
\begin{remark}
The proof of Theorem \ref{main1} is based on the equidistribution of the lifts of $C_d$ in the unit tangent bundle. For this reason, one may generalize Theorem \ref{main1} to any surfaces and any sequence of closed geodesics whose lifts become equidistributed on the unit tangent bundle.
\end{remark}

\subsection{Intersecting two closed geodesics}
We now consider the number of intersections between two closed geodesics when both vary.
\begin{theorem}\label{main}
The following estimate holds uniformly in $d_1,d_2>0$, and $0<\theta_1<\theta_2<\pi$ such that $\theta_2-\theta_1 \gg \left(d_1d_2\right)^{-\frac{25}{3072}}$
\[
\frac{I_{\theta_1,\theta_2}\left(C_{d_1},C_{d_2}\right)}{l\left(C_{d_1}\right)l\left(C_{d_2}\right)} = \frac{3}{\pi^2} \int_{\theta_1}^{\theta_2} \sin \theta d\theta  + O_{\epsilon}\left(\left(d_1d_2\right)^{-\frac{25}{6144}+\epsilon}\right).
\]
\end{theorem}
Note that if $\Gamma$ is co-compact, then the modular intersection kernel coincides with the \textit{intersection kernel} from \cite{Lalley} when $\theta=\pi$ and $\delta>0$ is sufficiently small. However, when $\Gamma \backslash \mathbb{H}$ is non-compact, then they are never the same; for instance, we have $K_\delta^{\alpha_1,\alpha_2}\left(g,g\right) = \Omega \left(y\right)$ as $y \to \infty$ (Proposition \ref{prop2}). In particular, $K_\delta^{\alpha_1,\alpha_2}$ is not a Hilbert---Schmit kernel, so the usual spectral theory does not apply. This is the main technical difficulty of dealing with the modular intersection kernel for non-compact quotients of $\mathbb{H}$. As it will be shown in the subsequent chapters, when both $\alpha_1$ and $\alpha_2$ are closed geodesics, $I_{0,\theta}\left(\alpha_1,\alpha_2\right)/\left(l\left(\alpha_1\right)l\left(\alpha_2\right)\right)$ is the integral of $\delta^{-2}K_\delta^{\theta_1,\theta_2}/\left(l\left(\alpha_1\right)l\left(\alpha_2\right)\right)$ over $\alpha_1 \times \alpha_2$. When $\alpha_1$ and $\alpha_2$ vary over closed geodesics of length $<L$, as $L \to \infty$, we expect that the integral converges to the integral of $\delta^{-2}K_\delta^{\theta_1,\theta_2}$ over $\Gamma \backslash S\mathbb{H} \times \Gamma \backslash S\mathbb{H}$, since $\alpha_1 \times \alpha_2$ becomes equidistributed in $\Gamma \backslash S\mathbb{H} \times \Gamma \backslash S\mathbb{H}$, as $L \to \infty$. However, unboundedness of the modular intersection kernel $K$ causes issues of interchanging the limit and the integral. In particular, the argument of \cite{PS} using intersection form does not apply in this case. Hence, in order to prove Theorem \ref{main}, we study the full spectral expansion of $K_\delta^{\alpha_1,\alpha_2}\left(g_1,g_2\right)$. This is similar to the existing work on the weight-$m$ Selberg's trace formula \cite{hejhal}, except that we have to deal with all weights simultaneously, and that the modular intersection kernel is not diagonalizable in general. We go over this carefully in Section \ref{pretrace}. Once the spectral expansion is obtained, the integral of $\delta^{-2}K_\delta^{\theta_1,\theta_2}$ over $\alpha_1 \times \alpha_2$ becomes a linear combination of the period integrals of the form
\[
\int_{\alpha_1} \phi_1 ds \times \int_{\alpha_2} \phi_2 ds.
\]
We may now use the same estimates that we use in order to prove the effective Duke's theorem to bound these, which leads to Theorem \ref{main}, generalizing \cite{PS} to a non-compact hyperbolic surface.

\section{The Modular Intersection Kernel}\label{prem}

\subsection{Parametrization}\label{para}
  Recall that $PSL_2\left(\RR\right)$ acts transitively on $\HH$ and on $S\HH$ with the fractional  transformations. For $g\in PSL_2\left(\RR\right),$  $z\in \HH$ and $u\in S\HH$ we write these actions by $gz$ and $gu$. We parameterize  the points of $\HH$ and  $S\HH$  with $x+iy$ and $\left(x+iy,\exp\left(i\theta\right)\right)$. Let
\[
\Pi \left( \left(x+iy,\exp\left(i\theta\right)\right) \right):=x+iy,
\]
be the projection map from $S\HH$ to $\HH$.

Fix $z_0=i$ and $u_0=\left(i,\exp\left(i\frac{\pi}{2}\right)\right)$. Let $g=naR_{\theta}\in PSL_2\left(\RR\right)$ be the Iwasawa decomposition where
\[
n=n\left(x\right)=\begin{pmatrix}
1& x
\\ 0 &1
\end{pmatrix},~
a=a\left(y\right)=\begin{pmatrix}
y^{\frac{1}{2}} &0 \\
0& y^{-\frac{1}{2}}
\end{pmatrix},~\text{ and }~
R_{\theta}=\begin{pmatrix}
\cos \theta &-\sin\theta
\\
\sin\theta &\cos\theta
\end{pmatrix}.
\]
Then we have
\[
gz_0=x+iy
\]
and
\[
gu_0=\left(x+iy,\exp\left(i\left(\frac{\pi}{2}+2\theta\right)\right)\right).
\]
For the rest of the paper, we identify $S\HH$ with $PSL_2\left(\RR\right)$ by sending $g\in PSL_2\left(\RR\right)$ to $gu_0$. We often use the following fact in our computation without mentioning.
\begin{proposition}
The image under $\gamma \in SL_2\left(\mathbb{R}\right)$ of the geodesic segment of length $\delta$ corresponding to $g = \left(x,\xi\right)$ is the geodesic segment of length $\delta$ corresponding to $\gamma g$.
\end{proposition}

We use the volume form given by $dV = \frac{dxdyd\theta}{y^2}$. The volume of $S\mathbb{X}$ is then $\frac{\pi^2}{3}$.

\subsection{Preliminary estimates}
We first recall here the definition of the modular intersection kernel described in the introduction. For $\delta>0$ and $\theta_1,\theta_2 \in \left(0,\pi\right)$, we define the integral kernel
\[
k_\delta^{\theta_1,\theta_2}:S\mathbb{H} \times S\mathbb{H} \to \mathbb{R}
\]
by
\[
k_\delta^{\theta_1,\theta_2}\left(\left(x_1,\xi_1\right), \left(x_2,\xi_2\right)\right) =1,
\]
if the geodesic segment of length $\delta$ on $\mathbb{H}$ from $x_1$ with the initial vector $\xi_1$ and the segment from $x_2$ with the initial vector $\xi_2$ intersect at an angle $\in \left(\theta_1,\theta_2\right)$, and $0$ otherwise. Here the angle of the intersection of geodesic segments $l_1$ and $l_2$ at $p \in l_1 \cap l_2$ is measured counterclockwise from $l_1$ to $l_2$. Under the identification $S\mathbb{H} \cong PSL_2\left(\mathbb{R}\right)$ from \S\ref{para}, we note here that
\[
k_\delta^{\theta_1,\theta_2}\left(gg_1, gg_2\right) = k_\delta^{\theta_1,\theta_2}\left(g_1,g_2\right)
\]
for any $g,g_1,g_2 \in PSL_2\left(\mathbb{R}\right)$.

Now for a given discrete subgroup $\Gamma \subset PSL_2\left(\mathbb{R}\right)$, we define the modular intersection kernel $K_\delta^{\theta_1,\theta_2}: \Gamma \backslash PSL_2\left(\mathbb{R}\right)\times \Gamma \backslash PSL_2\left(\mathbb{R}\right) \to \mathbb{R}$ by taking the average of $k_\delta^{\theta_1,\theta_2}$ over $\Gamma$:
\[
K_\delta^{\theta_1,\theta_2} \left(g_1,g_2\right) = \sum_{\gamma \in \Gamma}k_\delta^{\theta_1,\theta_2} \left(g_1,\gamma g_2\right).
\]
Note that when $\Gamma$ is co-compact, and $\delta>0$ is less than a half of the injectivity radius of $\Gamma \backslash \mathbb{H}$, we have $K_\delta^{\theta_1,\theta_2} \leq 1$. However, when $\Gamma\backslash \mathbb{H}$ is  non-compact, $K_\delta^{\theta_1,\theta_2} \left(g_1,g_2\right)$ becomes arbitrarily large near the diagonal $g_1=g_2$ as $y_1, y_2 \to \infty$. This is illustrated in the following proposition when $\Gamma = PSL_2(\mathbb{Z})$.
\begin{proposition}\label{prop2}
Fix $0<\theta <\pi$. Then for any $1 >\delta >0$, we have
\[
K_\delta^{0,\theta} \left(g,g\right) = \Omega_\theta\left(\delta y\right).
\]
\end{proposition}
\begin{proof}
Consider
\[
g= \left(Re^{i\left(\frac{\pi}{2}+\alpha\left(\delta\right)\right)},  e^{i\alpha\left(\delta\right)}\right) \in S \mathbb{H},
\]
where $\alpha\left(\delta\right)$ is chosen such that the geodesic segment
\[
\beta_g:=\{Re^{i\theta} ~:~ |\theta -\frac{\pi}{2}|<\alpha\left(\delta\right)\} \subset \mathbb{H}
\]
has length $\delta$. Note that the length of the segment does not depend on $R$ and that $\alpha\left(\delta\right) \sim \delta$ as $\delta \to 0$. From this, we infer that $\beta_g$ and $\beta_g+n$ with $0<n\ll R\delta$ intersect.

The angle of intersection is explicitly given by $2 \arcsin \frac{n}{R}$. So for all sufficiently small $0<\delta < \theta$, we have
\[
k_\delta^{\theta_1,\theta_2} \left(g, \begin{pmatrix}1 & n \\ 0 & 1\end{pmatrix}g\right) = 1,
\]
for $0 < n \ll R\delta$. This implies that
\[
K_\delta^{\theta_1,\theta_2} \left(g,g\right) \gg \delta R \gg \delta y.\qedhere
\]
\end{proof}

In view of Proposition \ref{prop2}, the following proposition provides a nice upper bound of the modular intersection kernel.

\begin{proposition}\label{prop1}
Let $\Gamma = PSL_2\left(\mathbb{Z}\right)$ and let $1>\delta>0$. Let $h$ be a compactly supported function on $S\mathbb{H}$, where we assume that $h\left(\left(\cdot,\xi\right)\right)$ is supported in $B_\delta\left(i\right)$ for any $\xi\in S^1$. Define $H:\Gamma \backslash S\mathbb{H}\times \Gamma \backslash S\mathbb{H}$ by
\[
H\left(g_1,g_2\right) = \sum_{\gamma \in \Gamma } h\left(g_1^{-1} \gamma g_2\right)
\]
for $g_1, g_2 \in \Gamma \backslash PSL_2\left(\mathbb{R}\right)$. Then for $g_i=\left(z_i,\xi_i\right)$ with $\textrm{dist}_{ \Gamma \backslash \mathbb{H}}\left(z_1,z_2\right) > 2\delta$, we have
\[
H\left(g_1,g_2\right) =0.
\]
When $y_1>0$ and $y_2>0$ are sufficiently large, we have
\[
H \left(g_1,g_2\right) \ll \delta \sqrt{y_1y_2} \|h\|_{L^\infty} .
\]
\end{proposition}
\begin{proof}
If $H >0$, then there exists $\gamma \in \Gamma $ such that
\[
h \left(g_1^{-1}\gamma g_2\right)>0.
\]
This implies that the balls of radius $\delta$ centered at $z_1$ and  $\gamma z_2$ intersect, hence $\mathrm{dist}_\mathbb{H}\left(z_1,\gamma z_2\right) < 2\delta$, which contradicts the assumption.

Now to prove the second estimate, we first note that when $y_2$ is sufficiently large, we have $y\left(\gamma g_2\right) <1$ unless $\gamma = \begin{pmatrix}1 &n\\ 0 & 1 \end{pmatrix}$. Therefore $h \left(g_1^{-1}\gamma g_2\right)>0$ only if $\gamma = \begin{pmatrix}1 &n\\ 0 & 1 \end{pmatrix}$. Note that $h \left(g_1^{-1}\gamma g_2\right)=1$ holds only if $\textrm{dist}_\mathbb{H} \left(z_1,n+z_2\right) < 2\delta$. This is equivalent to
\[
\mathrm{arccosh}\left(1+ \frac{\left(n+x_2-x_1\right)^2 + \left(y_1-y_2\right)^2}{y_1y_2}\right) < 2\delta,
\]
and so
\[
\left(n+x_2-x_1\right)^2 < y_1y_2\left(\cosh\left(2\delta\right)-1\right)-\left(y_1-y_2\right)^2 \leq y_1y_2\left(\cosh\left(2\delta\right)-1\right),
\]
from which we infer that there are at most $ \ll \delta \sqrt{y_1y_2}$ choices of $\gamma$ which makes $h \left(g_1,\gamma g_2\right)>0$.
\end{proof}

Now we analyze the modular intersection kernel when one variable is assumed to be contained in a compact set. We first note that if $\delta$ is less than half of the injectivity radius of $g_0$ in $\Gamma \backslash S\mathbb{H}$, then for each $g \in S\mathbb{H}$, there is at most one $\gamma \in \Gamma $ such that
\[
k_\delta^{\theta_1,\theta_2} \left(g_0,\gamma g\right) \neq 0.
\]
Therefore $K_\delta^{\theta_1,\theta_2} \left(g_0, \cdot \right)$ coincides with $k_\delta^{\theta_1,\theta_2} \left(g_0, \cdot\right)$ in the $2\delta$-neighborhood of $g_0$, which is a translation of $k_\delta^{\theta_1,\theta_2} \left(\left(i,i\right),\cdot\right)$ around $\left(i,i\right)$.
\begin{lemma}\label{lem1}
For $0< \theta_1< \theta_2 <\pi$, we have
\[
\int_\mathbb{H} k_\delta^{\theta_1,\theta_2} \left(\left(i,i\right),g\right)dV = \left(\cos \theta_1 - \cos \theta_2\right) \delta^2.
\]
Assume that $0<\delta <1$. Then for any $\varepsilon = o\left(\delta\right)$ and $\varepsilon = o\left(\theta_2-\theta_1\right)$ there exist a smooth majorant $M_\delta^{\theta_1,\theta_2}$ and a smooth minorant $m_\delta^{\theta_1,\theta_2}$, i.e.,
\[
0\leq m_\delta^{\theta_1,\theta_2} \leq k_\delta^{\theta_1,\theta_2}\left(\left(i,i\right),\cdot\right) \leq M_\delta^{\theta_1,\theta_2},
\]
such that
\[
\int m_\delta^{\theta_1,\theta_2} dV ~\text{ and }~\int M_\delta^{\theta_1,\theta_2} dV
\]
are both
\[
\left(\cos \theta_1 -\cos \theta_2\right) \delta^2 \left(1+O\left(\varepsilon\right)\right),
\]
and that
\[
\|m_\delta^{\theta_1,\theta_2}\|_{W^{k,\infty }}+\|M_\delta^{\theta_1,\theta_2}\|_{W^{k,\infty }} = O_k\left(\varepsilon^{-k}\right).
\]
\end{lemma}

\begin{proof}
Note that the action of the geodesic flow of time $t$  on $S\HH =PSL_2\left(\mathbb{R}\right)$ is the multiplication from the right by $a\left(e^t\right)$.
For given $\varphi \in \left(\theta_1,\theta_2\right)$, we describe the collection of $g \in PSL_2\left(\mathbb{R}\right)$ for which the corresponding geodesic segment of length $\delta$ intersects $\{iy~:~e^{\delta}>y>1\}$ transversally at angle $\varphi$. Note that this happens only when
\[
g a\left(e^{\frac{t_2}{2}}\right)=\begin{cases}a\left(e^{\frac{t_1}{2}}\right)R_{\frac{\varphi}{2}},
\\
a\left(e^{\frac{t_1}{2}}\right)R_{\frac{\varphi+\pi}{2}}.
\end{cases}
\]
for some $0< t_1, t_2 <\delta$. Hence
\[
g= \begin{cases}a\left(e^{\frac{t_1}{2}}\right)R_{\frac{\varphi}{2}}a\left(e^{-\frac{t_2}{2}}\right), \\
 a\left(e^{\frac{t_1}{2}}\right)R_{\frac{\varphi+\pi}{2}}a\left(e^{-\frac{t_2}{2}}\right).
 \end{cases}
\]

Consider $\Psi:AKA \to  PSL_2\left(\mathbb{R}\right)$ given by
\[
\left(t_1,\varphi,t_2\right) \mapsto a\left(e^{\frac{t_1}{2}}\right)R_{\frac{\varphi}{2}}a\left(e^{-\frac{t_2}{2}}\right)
\]
The determinant of the Jacobian of $\Psi$ is a nonzero multiple of $|\sin \varphi |$ (we refer the readers to Appendix \ref{app2} for the computation), and so this defines a local diffeomorphism away from $\varphi =0$ and $\pi$. Observe that $\Psi$ is injective away from $\varphi = 0$ and $\pi$. From this we infer that the support of $k_\delta^{\theta_1,\theta_2} \left(\left(i,i\right),g\right)$ is the image of the open box
\[
\{\left(t_1,\varphi,t_2\right)~:~ 0<t_1,t_2 <\delta,~ \theta_1 <\varphi <\theta_2~\text{ or }~\theta_1+\pi <\varphi <\theta_2+\pi\}
\]
under $\Psi$, and
\begin{multline*}
\int_\mathbb{H} k_\delta^{\theta_1,\theta_2} \left(\left(i,i\right),g\right)dV 
= \frac{1}{2} \int_{0}^{\delta}\int_{0}^\delta\int_{\theta_1}^{\theta_2}  |\sin\left(\varphi\right)| d\varphi dt_1 dt_2 +  \frac{1}{2} \int_{0}^{\delta}\int_{0}^\delta\int_{\theta_1+\pi}^{\theta_2+\pi}  |\sin\left(\varphi\right)| d\varphi dt_1 dt_2\\
=\left(\cos \theta_1 - \cos \theta_2\right) \delta^2,
\end{multline*}
where we used $dV = \frac{1}{2} |\sin \varphi| d\varphi dt_1dt_2$ (\eqref{app:1}).

Note that the support of $k_\delta^{\theta_1,\theta_2} \left(\left(i,i\right),\cdot\right)$ is an open set which has a piecewise smooth boundary. Therefore, under the assumption that $\varepsilon = o\left(\delta\right)$ and $\varepsilon = o\left(\theta_2-\theta_1\right)$, there exist smooth majorant and minorant whose $L^1$ norms are $ \left(\cos \theta_1 -\cos \theta_2\right) \delta^2 \left(1+O\left(\varepsilon\right)\right)$, and whose $k$-th derivatives are $O_k\left(\varepsilon^{-k}\right)$.
\end{proof}
As an immediate application, we have the following corollary.
\begin{corollary}
Fix a compact subset $C \subset \Gamma \backslash S\mathbb{H}$, and assume that $\delta$ is less than the half of the infimum of injectivity radius of $g \in C$ in $\Gamma \backslash S\mathbb{H}$. Then for any given compact geodesic segment $\beta \subset C$, and for any given $\varepsilon>0$ which is $o\left(\delta\right)$ and $o\left(\theta_2-\theta_1\right)$,
\[
\int_\beta K_\delta^{\theta_1,\theta_2}\left(s, \cdot\right) ds
\]
admits a smooth majorant $M_{\beta,\delta}^{\theta_1,\theta_2}$ and a smooth minorant $m_{\beta,\delta}^{\theta_1,\theta_2}$ such that
\[
\|m_{\beta,\delta}^{\theta_1,\theta_2}\|_{L^1}, \|M_{\beta,\delta}^{\theta_1,\theta_2}\|_{L^1} =  l\left(\beta\right)\left(\cos\theta_1-\cos \theta_2\right) \delta^2\left(1+O\left(\varepsilon\right)\right),
\]
and that
\[
\|m_{\beta,\delta}^{\theta_1,\theta_2}\|_{W^{k,\infty }}+\|M_{\beta,\delta}^{\theta_1,\theta_2i}\|_{W^{k,\infty }} = O_k\left(l\left(\beta\right)\varepsilon^{-k}\right).
\]
\end{corollary}

\subsection{Intersection numbers}
In this section, we prove formulas relating the number of intersections between two geodesics to the integral of the modular intersection kernel over the two geodesics.
\begin{lemma}\label{lemma1}
Let $\alpha_i=\{\alpha_i\left(t\right)~:~ t \in [0,l\left(\alpha_i\right))\}$ be closed geodesics in $\Gamma\backslash \mathbb{H}$ parameterized by the arc length, and let $\widetilde{\alpha}_i = \{\left(\alpha_i\left(t\right), \alpha_i'\left(t\right)\right)~:~ t \in [0,l\left(\alpha_i\right))\} \subset S\mathbb{H}$ be the lifts of $\alpha_i$ for $i=1,2$. Then for any $\delta>0$,
\[
I_{\theta_1,\theta_2}\left(\alpha_1,\alpha_2\right) = \frac{1}{\delta^2}\int_{\widetilde{\alpha}_2}\int_{\widetilde{\alpha}_1} K_{\delta}^{\theta_1,\theta_2} \left(s_1,s_2\right) ds_1 ds_2.
\]
\end{lemma}
\begin{remark}
For each $\alpha_i$, there are two choices of parameterization by the arc length, namely $\alpha_i\left(t\right)$ and $\alpha_i\left(-t\right)$, but the integral does not depend on the choice of the parameterizations.
\end{remark}
\begin{proof}
By abuse of notations, we think of each $\alpha_i$ with $t \in [0,l\left(\alpha_i\right))$ a geodesic segment in $\mathbb{H}$ and accordingly $\widetilde{\alpha}_i$ a corresponding curve in $S\mathbb{H}$. For a geodesic segment $\alpha \subset \mathbb{H}$ parameterized by $t\in [a,b]$, let $[\alpha]\subset \mathbb{H}$ be the bi-infinite geodesic $\{\alpha\left(t\right)~:~ t\in \mathbb{R}\}$ that contains $\alpha$. Then we express the integral as follows:
\begin{multline*}
\frac{1}{\delta^2}\int_{\widetilde{\alpha}_2}\int_{\widetilde{\alpha}_1} K_{\delta}^{\theta_1,\theta_2} \left(s_1,s_2\right) ds_1 ds_2
=\sum_{\gamma \in \Gamma} \frac{1}{\delta^2}\int_{\gamma \widetilde{\alpha}_2}\int_{ \widetilde{\alpha}_1} k_{\delta}^{\theta_1,\theta_2} \left(s_1,s_2\right) ds_1 ds_2\\
=\sum_{\gamma \in \Gamma / \Gamma_{[\alpha_2]}}\frac{1}{\delta^2}\int_{\gamma \widetilde{[\alpha_2]}}\int_{ \widetilde{\alpha}_1} k_{\delta}^{\theta_1,\theta_2} \left(s_1,s_2\right) ds_1 ds_2\\
= \sum_{\gamma \in \Gamma_{[\alpha_1]}\backslash \Gamma / \Gamma_{[\alpha_2]}}\sum_{\gamma' \in \Gamma_{[\alpha_1]}}
\frac{1}{\delta^2}\int_{\gamma' \gamma \widetilde{[\alpha_2]}}\int_{ \widetilde{\alpha}_1} k_{\delta}^{\theta_1,\theta_2} \left(s_1,s_2\right) ds_1 ds_2\\
= \sum_{\gamma \in \Gamma_{[\alpha_1]}\backslash \Gamma / \Gamma_{[\alpha_2]}}\frac{1}{\delta^2}\int_{ \gamma \widetilde{[\alpha_2]}}\int_{ \widetilde{[\alpha_1]}} k_{\delta}^{\theta_1,\theta_2} \left(s_1,s_2\right) ds_1 ds_2.
\end{multline*}
Here $\Gamma_{[\alpha_i]}$ is the stabilizer subgroup of $\Gamma$ with respect to $[\alpha_i]$.

Now because two geodesics in $\mathbb{H}$ may intersect at most once, for each intersection point $p \in \alpha_1 \cap \alpha_2$ on $\Gamma \backslash \mathbb{H}$, there exists a unique $\gamma \in \Gamma / \Gamma_{[\alpha_2]}$ such that $\alpha_1$ and $\gamma [\alpha_2]$ intersect at a lift of $p$. Also, because $[\alpha_1]$ is a disjoint union of $\gamma' \alpha_1$ with $\gamma' \in \Gamma_{[\alpha_1]}$, each $\{\gamma'\gamma~:~ \gamma' \in \Gamma_{[\alpha_1]}\}$ contains at most one $\gamma'\gamma$ such that $\gamma' \gamma [\alpha_2]$ intersects $\alpha_1$.

Therefore the intersections of $\alpha_1$ and $\alpha_2$ are in one-to-one correspondence with $\gamma \in \Gamma_{[\alpha_1]}\backslash \Gamma / \Gamma_{[\alpha_2]}$ such that $\gamma [\alpha_2]$ intersects $[\alpha_1]$. We complete the proof by observing that
\[
\int_{ \gamma \widetilde{[\alpha_2]}}\int_{ \widetilde{[\alpha_1]}} k_{\delta}^{\theta_1,\theta_2} \left(s_1,s_2\right) ds_1 ds_2 =1,
\]
if $[\alpha_1]$ and $\gamma[\alpha_2]$ intersect at an angle $\in \left(\theta_1,\theta_2\right)$, and $=0$ otherwise.
\end{proof}
Now let $\beta = \{\beta\left(t\right) ~:~ t \in [0,l\left(\beta\right))\}$ be a compact geodesic segment in $\Gamma \backslash \mathbb{H}$, and let $\alpha_2$ be a closed geodesic as before. Then
\[
\frac{1}{\delta^2}\int_{\widetilde{\alpha}_2}\int_{\widetilde{\beta}} K_{\delta}^{\theta_1,\theta_2} \left(s_1,s_2\right) ds_1 ds_2
\]
does not always give $I\left(\beta,\alpha_2\right)$. Instead, it is a weighted sum over the intersections of $\beta_0:=\{\beta\left(t\right) ~:~ t \in [0,l\left(\beta\right)+\delta)\}$ and $\alpha_2$. We prove the following.
\begin{lemma}\label{counting}
With the same notations as above, assume that $0<\delta < l\left(\beta\right)$ and that $\beta_0$ has no self intersection. For $0<\theta_1<\theta_2<\pi$, let $S\left(\beta_0, \alpha_2\right)_{\theta_1,\theta_2}$ be the set of intersections between $\beta_0$ and $\alpha_2$ where the intersection angle is $\in \left(\theta_1,\theta_2\right)$. Then we have
\[
\frac{1}{\delta^2}\int_{\widetilde{\alpha}_2}\int_{\widetilde{\beta}} K_{\delta}^{\theta_1,\theta_2} \left(s_1,s_2\right) ds_1 ds_2 = \sum_{p\in S\left(\beta_0, \alpha_2\right)_{\theta_1,\theta_2}} \min\left\{\frac{\beta^{-1}\left(p\right)}{\delta}, 1, \frac{l\left(\beta\right)+\delta-\beta^{-1}\left(p\right)}{\delta}\right\}.
\]
\end{lemma}
\begin{proof}
As in the proof of Lemma \ref{lemma1}, we first have
\begin{multline*}
\frac{1}{\delta^2}\int_{\widetilde{\alpha}_2}\int_{\widetilde{\beta}} K_{\delta}^{\theta_1,\theta_2} \left(s_1,s_2\right) ds_1 ds_2
=\sum_{\gamma \in \Gamma} \frac{1}{\delta^2}\int_{\gamma \widetilde{\alpha}_2}\int_{ \widetilde{\beta}} k_{\delta}^{\theta_1,\theta_2} \left(s_1,s_2\right) ds_1 ds_2\\
=\sum_{\gamma \in \Gamma / \Gamma_{[\alpha_2]}}\frac{1}{\delta^2}\int_{\gamma \widetilde{[\alpha_2]}}\int_{ \widetilde{\beta}} k_{\delta}^{\theta_1,\theta_2} \left(s_1,s_2\right) ds_1 ds_2.
\end{multline*}
Note that because we assumed that $\beta_0$ has no self-intersection, $p \in S\left(\beta_0, \alpha_2\right)_{\theta_1,\theta_2}$ is in one-to-one correspondence with $\gamma \in \Gamma / \Gamma_{[\alpha_2]}$ such that $\beta_0$ and $\gamma \widetilde{[\alpha_2]}$ intersect at $p$ at an angle $\in \left(\theta_1,\theta_2\right)$. We denote by $\gamma_p$ the $\gamma$ corresponding to $p$. Observe that
\[
\int_{\gamma \widetilde{[\alpha_2]}}\int_{ \widetilde{\beta}} k_{\delta}^{\theta_1,\theta_2} \left(s_1,s_2\right) ds_1 ds_2 =0,
\]
if $\gamma \widetilde{[\alpha_2]} \cap \beta_0 = \emptyset$. So it is sufficient to prove that
\[
\frac{1}{\delta^2}\int_{\gamma_p \widetilde{[\alpha_2]}}\int_{ \widetilde{\beta}} k_{\delta}^{\theta_1,\theta_2} \left(s_1,s_2\right) ds_1 ds_2=\min\left\{\frac{\beta^{-1}\left(p\right)}{\delta}, 1, \frac{l\left(\beta\right)+\delta-\beta^{-1}\left(p\right)}{\delta}\right\}.
\]
This follows by observing that $k_{\delta}^{\theta_1,\theta_2} \left(\left(\beta\left(t_1\right),\beta'\left(t_1\right)\right),\left(\gamma_p\alpha_2\left(t_2\right),\left(\gamma_p\alpha_2\right)'\left(t_2\right)\right)\right)=1$ for
\[
\left(t_1,t_2\right) \in \left(\beta^{-1}\left(p\right)-\delta, \beta^{-1}\left(p\right)\right) \times \left(\alpha_2^{-1}\left(p\right)-\delta, \alpha_2^{-1}\left(p\right)\right),
\]
and $0$ otherwise, whereas the integral over $\widetilde{\beta}$ is over the range $t_1 \in (0,l(\beta))$.
\end{proof}

\section{Spectral theory}
\subsection{Spectral expansion}\label{S1}
We first go over the spectral decomposition of $L^2\left(S\mathbb{X}\right)$. Readers may find more details on the subject from \cite{kubota} and \cite{lang}. On $G=PSL_2\left(\mathbb{R}\right)$, there is a differential operator of order $2$ that commutes with $G$ action:
\[
\Omega = y^2 \partial_x^2 + y^2\partial_y^2 + y \partial_x \partial_\theta,
\]
which is called the Casimir operator. An equivariant eigenfunctions of $\Omega$ is a function $f \in C^\infty (S\mathbb{X})$ that satisfies
\[
\Omega f = \lambda f
\]
for some $\lambda \in \mathbb{R}$, and
\begin{equation}\label{weight1}
f\left(gR_\theta\right) = e^{-im\theta} f\left(g\right)
\end{equation}
for some $m \in 2\mathbb{Z}$. We say that a function has weight $m$ if it satisfies \eqref{weight1}.

Each irreducible (cuspidal) sub-representation of the right regular representation
\[
\rho_g : f\left(h\right) \mapsto f\left(h g\right)
\]
on $L^2\left(S\mathbb{X}\right)$ is generated by an equivariant eigenfunction of $\Omega$.

We let $\mathbf{E}^+$ and $\mathbf{E}^-$ to be the raising and lowering operator acting on equivariant functions on $L^2\left(S\mathbb{X}\right)$, which are given by \cite{dima}
\begin{equation}\label{dma}
\begin{split}
\mathbf{E}^+ &= e^{-2i\theta} \left(2iy \partial_x + 2y \partial_y + i \partial_\theta\right), \text{ and}\\
\mathbf{E}^- &= e^{2i\theta} \left(2iy \partial_x - 2y \partial_y + i \partial_\theta\right).
\end{split}
\end{equation}
Note that $\mathbf{E}^+$ (resp. $\mathbf{E}^-$) maps a weight $m$ eigenfunction of $\Omega$ to a weight $m+2$ (resp. $m-2$) eigenfunction of $\Omega$.

For an even integer $m$ let
\[
\psi_{s,m}\left(g\right)=y^{s}e^{-im\theta}.
\]
Note that $\psi_{s,m}$ is invariant under the action of the unipotent upper triangular matrices. The weight-$m$ Eisenstein series is then given by
\[
E_{m}\left(g,s\right)=\sum_{\gamma\in \Gamma_{\infty}\backslash\Gamma}\psi_{s,m}\left(\gamma g\right),
\]
where $\Gamma_\infty = \left\{\begin{pmatrix}1 & n \\  & 1\end{pmatrix}~:~ n \in \mathbb{Z}\right\}$ is the stabilizer subgroup of $\Gamma$ with respect to the cusp $i\infty$. Although the right-hand side of the equation is absolutely convergent only for $\mathrm{Re}(s) >1$, the weight-$m$ Eisenstein series has a meromorphic continuation to the entire complex plane.

Let $\Theta$ be the closure of
\[
\left\{\int_{-\infty}^\infty h\left(t \right)E_{m}\left(g,\frac{1}{2}+it \right)dt ~:~ h\left(t \right) \in C_0^\infty \left(\mathbb{R}\right),~m\in 2\mathbb{Z}   \right\}
\]
in $L^2 \left(S\mathbb{X}\right)$, and let
\[
L_{\text{cusp}}^2 \left(S\mathbb{X}\right)= \{f \in L^2\left(S\mathbb{X}\right)~:~ \int_{0}^1 f\left(n\left(x\right)g\right)dx =0 \text{ for almost every } g\in S\mathbb{X}\}
\]
be the space of cusp forms. Then we have the decomposition
\[
L^2\left(S\mathbb{X}\right) = \langle\{1\} \rangle \oplus \Theta \oplus L_{\text{cusp}}^2 \left(S\mathbb{X}\right),
\]
where $\langle\{1\} \rangle$ is the subspace spanned by a constant function.

We express the cuspidal subspace as a direct sum of subspaces generated by Maass forms and modular forms as in \cite[(1.10)]{LRS},
\[
L_{\text{cusp}}^2 \left(S\mathbb{X}\right) = \sum_{j=1}^\infty W_{\pi_j^0} \bigoplus \sum_{m\geq 12} \sum_{j=1}^{d_m} \left(W_{\pi_j^m}\oplus W_{\pi_j^{-m}}\right),
\]
where each $W_{\pi_j^m}$ corresponds to a $G$ and Hecke irreducible subspace of a right regular representation on $L_{\text{cusp}}^2$. Here $d_m$ is the dimension of the space of holomorphic cusp forms of weight $m$ for $PSL_2\left(\mathbb{Z}\right)$. Each $\pi_j^0$ corresponds to a Maass---Hecke cusp form which we denote by $\phi_j^0$. For $m> 0$, $\pi_j^m$ corresponds to a holomorphic Hecke cusp form $\phi_j^m$. We identify a weight $m$ function on $\Gamma \backslash \mathbb{H}$
\[
f\left(\gamma z\right) = \left(cz+d\right)^m f\left(z\right) ~\text{ for }~ \begin{pmatrix}a & b \\ c& d\end{pmatrix}=\gamma \in \Gamma
\]
with a weight $m$ $\Gamma$-invariant function $F$ on $PSL_2\left(\mathbb{R}\right)$ via
\begin{equation}\label{dict}
F\left(g\right) = y^{\frac{m}{2}} f\left(z\right) e^{-im\theta}.
\end{equation}
When $m\geq 0$, viewing $\phi_j^m$ as a function on $S\mathbb{X}$, each $W_{\pi_j^m}$ is spanned by
\[
\ldots,~\left(\mathbf{E}^-\right)^3 \phi_j^m,~ \left(\mathbf{E}^-\right)^2 \phi_j^m,~\mathbf{E}^- \phi_j^m,~\phi_j^m,~ \mathbf{E}^+\phi_j^m,~\left(\mathbf{E}^+\right)^2 \phi_j^m,~ \left(\mathbf{E}^+\right)^3 \phi_j^m, \ldots
\]
Note that when $m>0$, $\mathbf{E}^- \phi_j^m =0$.

For $m<0$, we set
\[
W_{\pi_j^{-m}}=\overline{W_{\pi_j^m}} = \{\bar{f}~:~ f\in W_{\pi_j^m} \}.
\]
Now let
\[
U_{\pi_j^0}=W_{\pi_j^0}, \text{ and } U_{\pi_j^m}= W_{\pi_j^m}\oplus W_{\pi_j^{-m}},
\]
when $m>0$. We specify an orthonormal basis of each $U_{\pi_j^m}$ as follows.

\textbf{The Maass cusp form case $m=0$.} Let $ - \left(\frac{1}{4}+t_j^2\right)$ be the Laplacian eigenvalue of $\phi_j^0$\footnote[2]{Formally, it is the eigenvalue of the Laplace---Beltrami operator on $\mathbb{X}$ that corresponds to $\phi_j^0$.}, for some real $t_j$. We set $\phi_{j,0}^0 = \phi_j^0$, and define $\phi_{j,l}^0$ for $l\in 2\mathbb{Z}$ inductively by
\begin{equation}\label{phi1}
\begin{split}
\mathbf{E}^- \phi_{j,l}^0 &= \left(l+1-2it_j\right)\phi_{j,l-2}^0, \text{ and}\\
\mathbf{E}^+ \phi_{j,l}^0 &= \left(l+1+2it_j\right)\phi_{j,l+2}^0.
\end{split}
\end{equation}

\textbf{The holomorphic Hecke cusp form case $m>0$.} We set $\phi_{j,m}^m = \phi_j^m$ and $\phi_{j,-m}^m = \overline{\phi_j^m}$, and define $\phi_{j,l}^m$ for $l\in 2\mathbb{Z}$ inductively by
\begin{equation}\label{phi2}
\begin{split}
\mathbf{E}^- \phi_{j,l}^m &= \left(l-m\right)\phi_{j,l-2}^m, \text{ and}\\
\mathbf{E}^+ \phi_{j,l}^m &= \left(l+m\right)\phi_{j,l+2}^m.
\end{split}
\end{equation}
Finally, note that we have the following relation among the weight $m$ Eisenstein series.
\begin{align*}
\mathbf{E}^- E_{m}\left(g,\frac{1}{2}+it\right) &= \left(m+1-2it\right)E_{m-2}\left(g,\frac{1}{2}+it\right),\text{ and}\\
\mathbf{E}^+ E_{m}\left(g,\frac{1}{2}+it\right) &= \left(m+1+2it\right)E_{m+2}\left(g,\frac{1}{2}+it\right).
\end{align*}
With these notations, we have
\begin{proposition}\label{exp}
Let $f \in L^2 \left(S\mathbb{X}\right)$. Then we have
\begin{multline*}
f\left(g\right) = \frac{3}{\pi^2}\int_{S\mathbb{X}} f\left(g_1\right) dg_1  +\sum_{\substack{m\geq 0\\ 2|m}} \sum_{j=1}^{d_m} \sum_{\substack{l \in 2\mathbb{Z}\\ |l|\geq m}}\left\langle f,\phi_{j,l}^m\right\rangle_{S\mathbb{X}}\phi_{j,l}^m\left(g\right)\\
+\sum_{m \in 2\mathbb{Z}}\frac{1}{4\pi}\int_{-\infty}^\infty \left\langle f,E_m\left(\cdot, \frac{1}{2}+it\right)\right\rangle_{S\mathbb{X}} E_m \left(g, \frac{1}{2}+it\right) dt,
\end{multline*}
where we set $d_0 = +\infty$.
\end{proposition}

\section{Effective Equidistribution}\label{Eisper}
\subsection{Invariant linear form}\label{lrs}
Define $\mu_d$ to be the integral over discriminant $d$ oriented closed geodesics on $S\mathbb{X}$,
\[
\mu_d\left(F\right) := \int_{\mathscr{C}_d} F\left(s\right)ds= \sum_{\text{disc}\left(q\right)=d}\int_{C\left(q\right)} F\left(s\right)ds.
\]
where $C\left(q\right)\subset S\mathbb{X}$ is the oriented closed geodesic associated to the binary quadratic form $q$ \cite[2.3]{LRS}.
Then for any $F \in U_{\pi_j^m}$, we have
\[
\mu_d\left(F\right) = \mu_d\left(\phi_j^m\right) \eta_j^m\left(F\right)
\]
for some linear form $\eta_j^m$ on $U_{\pi_j^m}$ invariant under the diagonal action \cite[\S3.7.1]{LRS}, which we describe below following  \cite[\S3.2]{LRS}. (Note that the parameter $s$ in \cite{LRS} is replaced by $2it$ in this article for consistency.)

\textbf{The Maass cusp form case $m=0$.} Let $\phi_{j,l}^0$ be the Maass form defined by \eqref{phi1}. When $4|l$ and $l\geq 4$, we have
\begin{equation}\label{period:1}
\eta_j^0 \left(\phi_{j,l}^0\right) = \eta_j^0 \left(\phi_{j,-l}^0\right) = \frac{\left(1-2it_j\right)\left(5-2it_j\right)\cdots\left(l-3-2it_j\right)}{\left(3+2it_j\right)\left(7+2it_j\right)\cdots\left(l-1+2it_j\right)},
\end{equation}
and $\eta_j^0 \left(\phi_{j,l}^0\right)$ is identically $0$ if $l\equiv 2 \pmod{4}$. Note that $\{\phi_{j,l}^0\}_{l\in 2\mathbb{Z}}$ is an orthogonal basis of $U_{\pi_j^0}$, and normalized so that,
\[
\|\phi_{j,l}^0\|_{L^2} =  \|\phi_j^0\|_{L^2}.
\]

\textbf{The holomorphic Hecke cusp form case $m>0$.}
Let $\phi_{j,l}^m$ be the holomorphic Hecke cusp form defined by \eqref{phi2}. When $m\equiv 2 \pmod{4}$, $\eta_j^m$ is identically $0$.

When $m\equiv 0 \pmod{4}$, for $l \geq 4$ with $4|l$,
\begin{equation}\label{period:2}
\eta_j^m \left(\phi_{j,m+l}^m\right) = \eta_j^m \left(\phi_{j,-m-l}^m\right) = \frac{1 \cdot 3 \cdot 5 \cdots \left(\frac{l}{2}-1\right)}{\left(m+1\right)\left(m+3\right) \cdots \left(m+\frac{l}{2}-1\right)},
\end{equation}
and $\eta_j^m \left(\phi_{j,m+l}^m\right)$ vanishes for $l \equiv 2 \pmod{4}$.

Note that $\{\phi_{j,l}^m\}_{l\in 2\mathbb{Z},~|l|\geq m}$ is an orthogonal basis of $U_{\pi_j^m}$, and normalized so that
\[
\|\phi_{j,l}^m\|_{L^2} =  \|\phi_j^m\|_{L^2}.
\]
for $l\in 2\mathbb{Z},~|l|\geq m$.

\textbf{Eisenstein series case.}
By the above identities and following \cite[Section 3]{LRS}, we have
 \[
\mu_d\left(E_{m}\left(g,\frac{1}{2}+it\right)\right)= \eta \left(m,t\right) \mu_d\left(E_{0}\left(g,\frac{1}{2}+it\right)\right),
\]
 where  for $m\geq 4$ such that $4|m$,
\begin{equation}\label{period:3}
\eta \left(m,t\right) =\eta \left(-m,t\right) = \frac{\left(1-2it\right)\left(5-2it\right)\cdots\left(2m-3-2it\right)}{\left(3+2it\right)\left(7+2it\right)\cdots\left(2m-1+2it\right)},
\end{equation}
and $\eta \left(m,t\right)$ is identically $0$ if $m\equiv 2 \pmod{4}$.

\subsection{Period integrals}
\subsubsection{Holomorphic cusp forms}
In this section, we give an upper bound on the period integrals of holomorphic forms. We
first use the results of Shintani to relate the period integrals of holomorphic cusp forms to the Fourier coefficients of half integral holomorphic forms.  We then apply the result of Kohnen and Zagier~\cite{KohnenZagier} which gives an explicit version of the Waldspurger's formula for the Fourier coefficients of half integral  holomorphic forms. An upper bound on these period integrals is deduced by using the subconvexity bounds on the central value of the $L$-functions and the Ramanujan bound on the Fourier coefficients of holomorphic modular forms.

Note that $c(d)$ is identically zero when $m\equiv 2 \pmod{4}$, and so we assume that $4|m$. Let $\hat{\phi}_j^{m}$ be a normalization of the Hecke holomorphic cusp form $\phi_j^m$ of weight $m$ such that $a_1=1$. Let
\[
c\left(d\right):=\sum_{\text{disc}\left(q\right)=d}\int_{C\left(q\right)} \hat{\phi}_j^{m}\left(z\right)q\left(z,1\right)^{\frac{m}{2}-1}dz,
\]
where $ \hat{\phi}_j^{m}\left(z\right)$ is the associated  holomorphic modular form defined on the upper half plane and the integration is on the upper half plane \eqref{dict}.
By~\cite[equation (2.4) p.14]{LRS}, we have
\begin{equation}\label{rand1}
|c\left(d\right)|=|d|^{\frac{m}{4}-\frac{1}{2}}|\mu_d\left(\hat{\phi}_j^{m}\right)|.
\end{equation}
Let
\[
\theta\left(z,\hat{\phi}_j^{m}\right):=\sum_{d\geq 1}c\left(d\right)e\left(dz\right).
\]
By~\cite[Theorem 2]{Shintani}, $\theta\left(z,\phi_j^{m}\right)$ is a Hecke holomorphic cusp form of weight $\frac{m+1}{2}$ and level $\Gamma_0\left(4\right)$. By  \cite[(6.2), p.37]{LRS}, we have the following explicit version of Rallis inner product formula
\[
\left\langle \theta\left(\hat{\phi}_j^{m}\right),\theta\left(\hat{\phi}_j^{m}\right)\right\rangle=\frac{\left(\frac{m}{2}-1\right)!}{2^{m}\pi^{\frac{m}{2}}}L\left(\frac{1}{2},\phi_j^{m}\right)\left\langle\hat{\phi}_j^{m},\hat{\phi}_j^{m}  \right\rangle.
\]
Suppose that $d=Db^2$ with $D$ a fundamental discriminant. By~\cite[Theorem 1]{KohnenZagier}, for $D$ a fundamental discriminant with $D>0$ and $4|m$, we have
\[
\frac{c\left(D\right)^2}{\left\langle \theta\left(\hat{\phi}_j^{m}\right),\theta\left(\hat{\phi}_j^{m}\right)\right\rangle}=\frac{\left(\frac{m}{2}-1\right)!}{\pi^{\frac{m}{2}}}D^{\frac{m-1}{2}}\frac{L\left(\frac{1}{2},\phi_j^{m}\otimes\chi_D\right)}{\left\langle \hat{\phi}_j^{m},\hat{\phi}_j^{m}\right\rangle},
\]
which implies that
\[
|c\left(D\right)|= D^{\frac{m-1}{4}}\frac{\left(\frac{m}{2}-1\right)!}{2^{\frac{m}{2}}\pi^{\frac{m}{2}}}\left(L\left(\frac{1}{2},\phi_j^{m}\right)L\left(\frac{1}{2},\phi_j^{m}\otimes\chi_D\right)\right)^{\frac{1}{2}}.\]
By using the Ramanujan bound on the Fourier coefficients of integral weight cusp forms and the above, we have
\[
|c\left(d\right)|\ll_{\epsilon} b^{\frac{m-1}{2}+\epsilon}|c\left(D\right)| \ll_{\epsilon}d^{\frac{m-1}{4}+\epsilon}\frac{\left(\frac{m}{2}-1\right)!}{2^{\frac{m}{2}}\pi^{\frac{m}{2}}}\left(L\left(\frac{1}{2},\phi_j^{m}\right)L\left(\frac{1}{2},\phi_j^{m}\otimes\chi_D\right)\right)^{\frac{1}{2}},
\]
and so
\[
|\mu_d\left(\hat{\phi}_j^{m}\right)|\ll_{\epsilon} |d|^{\frac{1}{4}+\epsilon}\frac{\left(\frac{m}{2}-1\right)!}{2^{\frac{m}{2}}\pi^{\frac{m}{2}}}\left(L\left(\frac{1}{2},\phi_j^{m}\right)L\left(\frac{1}{2},\phi_j^{m}\otimes\chi_D\right)\right)^{\frac{1}{2}},
\]
by \eqref{rand1}.

We now use the convexity bound
\[
L\left(\frac{1}{2},\phi_j^{m}\right)\ll_{\epsilon} m^{\frac{1}{2}+\epsilon},
\]
and the subconvexity bound \cite[Theorem 1]{weyl}
\[
L\left(\frac{1}{2},\phi_j^{m}\otimes\chi_D\right)\ll_{\epsilon} m^{\frac{75+12\theta}{16}}D^{\frac{1}{2}-\frac{1}{8}\left(1-2\theta\right)+\epsilon},
\]
where $\theta=\frac{7}{64}$ is the best exponent toward Ramanujan conjecture for Maass forms, to see that
\[
|\mu_d\left(\hat{\phi}_j^{m}\right)|\ll_{\epsilon} d^{\frac{1}{4}+\epsilon}\frac{\left(\frac{m}{2}-1\right)!}{2^{\frac{m}{2}}\pi^{\frac{m}{2}}}m^{2.64}D^{\frac{1}{4}-\frac{25}{512}}.
\]
 It is well-known that
\[
\left\langle \hat{\phi}_j^{m},\hat{\phi}_j^{m}\right\rangle= \frac{\Gamma\left(m\right)}{\left(4\pi\right)^{m}}L\left(1,\text{sym}^2 \phi_j^{m}\right)
\]
up to a constant. Hence, by Stirling's approximation
\begin{equation}\label{per:1}
|\mu_d\left(\phi_j^{m}\right)|\ll_{\epsilon} d^{\frac{1}{4}+\epsilon}m^{2.9}D^{\frac{1}{4}-\frac{25}{512}}\ll d^{\frac{1}{2}- \frac{25}{512}+\epsilon}m^{2.9}.
\end{equation}

\subsubsection{Maass forms}
In this section, we give an upper bound on the period integrals of Maass forms. We first recall some results of Katok and Sarnak~\cite{Katok} that generalize the work of Shintani~\cite{Shintani} to Maass forms and related the period integrals to the Fourier coefficients of half integral Maass forms. Then we use an explicit version of  the Waldspurger  formula~\cite{BaruchMao} and  give a non-trivial bound on these period integrals by using the subconvexity bound on the central value of the $L$-functions and the best bound toward Ramanujan conjecture for Maass forms.

Let $\phi_j^0$ be a Hecke---Maass form with $\left\langle\phi_j^0,\phi_j^0  \right\rangle=1$ and with the Laplacian eigenvalue $-\left(\frac{1}{4}+t_j^2\right)$. For $d>0$, let
\[
\rho\left(d\right):=\frac{1}{\sqrt{8}\pi^{\frac{1}{4}}d^{\frac{3}{4}}}\sum_{disc\left(q\right)=d}\int_{C\left(q\right)}\phi_j^0 ds
\]
be the associated period integral, and let
\[
\theta\left(\left(u+iv\right),\phi_j^0\right):=\sum_{d\neq 0}\rho\left(d\right)W_{\frac{\text{sgn}\left(d\right)}{4},\frac{it_j}{2}}\left(4\pi|d|v\right)e\left(du\right),
\]
where $W_{\frac{\text{sgn}\left(d\right)}{4},\frac{it_j}{2}}$ is the usual Whittaker function. Here
 $\rho\left(d\right)$ for $d<0$ is the sum of $\phi_j^0$ over the CM points with the discriminant $d$ appropriately normalized (see ~\cite[p.197]{Katok} or \cite[section 3.3]{leastprime} for a detailed discussion.)

Note from \cite{Katok} that $\theta\left(\left(u+iv\right),\phi_j^0\right)$ is a weight-$\frac{1}{2}$ Hecke---Maass form with the Laplacian eigenvalue $-\left(\frac{1}{4}+\frac{t_j^2}{4}\right)$.
By \cite[(5.6), p.224]{Katok} or \cite[(6.4), p.38]{LRS}, we have the following version of the Rallis inner product formula
\[
\left\langle\theta\left(\phi_j^0\right),\theta\left(\phi_j^0\right)  \right\rangle=\frac{3}{2}\Lambda\left(\frac{1}{2},\phi_j^0\right),
\]
where
\[
\Lambda\left(s,\phi_j^0\right)=\pi^{-s}\Gamma\left( \frac{s+it_j}{2} \right) \Gamma\left(\frac{s-it_j}{2}  \right)L\left(s,\phi_j^0\right)
\]
is the completed $L$-function.

By an explicit form of  Waldspurger  formula~\cite[Theorem 1.4]{BaruchMao}, and the best exponent toward the Ramanujan conjecture~\cite[Corollary 6.1]{RadziwillLester}, we have
\[
\frac{\rho\left(d\right)}{\left\langle\theta\left(\phi_j^0\right),\theta\left(\phi_j^0\right)  \right\rangle^{\frac{1}{2}}}\ll_{\epsilon}\frac{1}{\sqrt{|d|}}\left(\frac{L\left(\frac{1}{2},\phi_j^0\otimes\chi_D\right)}{L\left(1,\text{sym}^2 \phi_j^0\right)} \right)^{\frac{1}{2}} b^{\frac{7}{64}+\epsilon}|t_j|^{-\frac{\text{sgn}\left(d\right)}{4}}e^{\frac{\pi|t_j|}{4}},
\]
where $d=Db^2$ with $D$ a fundamental discriminant. Note from Stirling's formula that
\[
\Gamma\left( \frac{\frac{1}{2}+it_j}{2} \right) \Gamma\left(\frac{\frac{1}{2}-it_j}{2}  \right)\ll |t_j|^{-\frac{1}{2}}e^{-\frac{\pi|t_j|}{2}},
\]
from which we infer that
\begin{align*}
\mu_d\left(\phi_j^0\right)
&\ll d^{\frac{3}{4}}|\rho\left(d\right)|
\\&\ll_{\epsilon} d^{\frac{1}{4}} \left(\Lambda\left(\frac{1}{2},\phi_j^0\right)\right)^{\frac{1}{2}}\left(\frac{L\left(\frac{1}{2},\phi_j^0\otimes\chi_D\right)}{L\left(1,\text{sym}^2 \phi_j^0\right)} \right)^{\frac{1}{2}} b^{\frac{7}{64}+\epsilon}|t_j|^{-\frac{\text{sgn}\left(d\right)}{4}}e^{\frac{\pi|t_j|}{4}}
\\&\ll_{\epsilon}  d^{\frac{1}{4}} \left(L\left(\frac{1}{2},\phi_j^0\right)L\left(\frac{1}{2},\phi_j^0\otimes\chi_D\right)\right)^{\frac{1}{2}}b^{\frac{7}{64}+\epsilon}|t_j|^{-\left(\frac{\text{sgn}\left(d\right)+1}{4}\right)+\epsilon}.
\end{align*}
We now use the convexity bound,
\[
L\left(\frac{1}{2},\phi_j^0\right)\ll_{\epsilon} |t_j|^{\frac{1}{2}+\epsilon},
\]
and the subconvexity bound~\cite[Theorem 1]{weyl},
\[
L\left(\frac{1}{2},\phi_j^0\otimes\chi_D\right)\ll_{\epsilon} |t_j|^{\frac{31+4\theta+\epsilon}{16}}D^{\frac{1}{2}-\frac{1}{8}\left(1-2\theta\right)+\epsilon},
\]
to conclude that
\begin{equation}\label{per:2}
\mu_d\left(\phi_j^0\right)\ll_{\epsilon}  d^{\frac{1}{4}+\epsilon} |t_j|^{\frac{3}{4}} b^{\frac{7}{64}+\epsilon} D^{\frac{1}{4}-\frac{25}{512}} \ll d^{\frac{1}{2}-\frac{25}{512}+\epsilon} |t_j|^{\frac{3}{4}}.
\end{equation}
\subsubsection{Eisenstein series}
For a non-square integer $d \equiv 0,1 \pmod{4}$, let $d=Db^2$ where $D$ is a fundamental discriminant. Then we have the following explicit formula for the period integral of the Eisenstein series \cite[p.282]{zagierid}\footnote[2]{When $b=1$, this is a classical result due to Hecke \cite[p.88]{Siegeladv}.}:
\begin{equation}\label{classnum}
 \mu_d\left(E_{0}\left(\cdot,s\right)\right)= \frac{\Gamma(\frac{s}{2})^2d^{\frac{s}{2}}L(s,d) }{\Gamma(s)\zeta(2s) },
\end{equation}
where
\begin{equation}\label{l1d}
L(s,d)= L(s,\chi_D)\left(\sum_{a|b} \mu(a) \left(\frac{D}{a} \right) a^{-s} \sigma_{1-2s}\left(\frac{b}{a}\right)  \right).
\end{equation}
Here $L(s,\chi_D)$ is the Dirichlet $L$-function attached to the quadratic Dirichlet character $\chi_D (\cdot) = \left(\frac{D}{\cdot} \right)$, $\mu(\cdot)$ is the M\"obius function, and  $\sigma_{v}(\cdot)=\sum_{a|\cdot}a^{v}$ is the divisor function.

Now assume that $s=\frac{1}{2}+it$ for some $t\in \mathbb{R}$. By Stirling's formula, we have
\[
\frac{\Gamma(\frac{s}{2})^2}{\Gamma(s)}\ll |t|^{-\frac{1}{2}}.
\]
By the zero free region of $\zeta(2s)$ around $2s=1+2it$, we have
\[
|\zeta(2s)|\gg_\epsilon t^{-\epsilon}.
\]
We also have the convexity bound
\[
\zeta\left(s\right) \ll |t|^{\frac{1}{4}},
\]
and we know from \cite{hybrid} that
\[
L\left(\frac{1}{2}+it,\chi_D\right) \ll_\epsilon \left(\left(|t|+1\right)D\right)^{\frac{3}{16}+\epsilon}.
\]
Finally, observe that we have
\[
\sum_{a|b} \mu(a) \left(\frac{D}{a} \right) a^{-s} \sigma_{1-2s}\left(\frac{b}{a}\right)  \ll_\epsilon d^{\epsilon}.
\]
Combining all these estimates, we deduce the following estimate from \eqref{classnum} for $s=\frac{1}{2}+it$:
\begin{equation}\label{per:3}
\mu_d\left(E_{0}\left(\cdot,s\right)\right) \ll_\epsilon d^{\frac{1}{2}-\frac{1}{16}+\epsilon}.
\end{equation}

\subsection{Proof of Theorem \ref{effect}}

For any compactly supported smooth function $F \in C_0^\infty \left(S\mathbb{X}\right)$, recall from Proposition \ref{exp} that we have
\begin{multline*}
F\left(g\right) = \frac{3}{\pi^2}\int_{S\mathbb{X}} F\left(g_1\right) dg_1  +\sum_{\substack{m\geq 0\\ 2|m}} \sum_{j=1}^{d_m} \sum_{\substack{l \in 2\mathbb{Z}\\ |l|\geq m}}\left\langle F,\phi_{j,l}^m\right\rangle_{S\mathbb{X}}\phi_{j,l}^m\left(g\right)\\
+\sum_{m \in 2\mathbb{Z}}\frac{1}{4\pi}\int_{-\infty}^\infty \left\langle F,E_m\left(\cdot, \frac{1}{2}+it\right)\right\rangle_{S\mathbb{X}} E_m \left(g, \frac{1}{2}+it\right) dt,
\end{multline*}
and so from the discussion of Section \ref{lrs}, we have
\begin{multline*}
\mu_d\left(F\right) = \mu_d\left(\frac{3}{\pi^2}\right)\int_{S\mathbb{X}} F\left(g\right) dg  +\sum_{\substack{m\geq 0\\ 4|m}} \sum_{j=1}^{d_m}\mu_d\left(\phi_{j}^m\right) \sum_{\substack{l \in 4\mathbb{Z}\\ |l|\geq m}}\left\langle F,\phi_{j,l}^m\right\rangle_{S\mathbb{X}}\eta_j^m\left(\phi_{j,l}^m\right)\\
+\sum_{m \in 4\mathbb{Z}}\frac{1}{4\pi}\int_{-\infty}^\infty \left\langle F,E_m\left(\cdot, \frac{1}{2}+it\right)\right\rangle_{S\mathbb{X}} \eta\left(m, \frac{1}{2}+it\right)\mu_d\left(E_0 \left(\cdot, \frac{1}{2}+it\right)\right) dt.
\end{multline*}
Firstly, we have from \eqref{period:1}, \eqref{period:2}, and \eqref{period:3} that $\eta_j^m\left(\phi_{j,l}^m\right)$ and $ \eta\left(m, \frac{1}{2}+it\right)$ are both $O\left(1\right)$. Note by successive integration by parts and Cauchy---Schwarz inequality, we have for all $N\geq 1$,
\[
 \langle F, \phi_{j,l}^m \rangle \ll_N \left(|l|^2 +1\right)^{-N} \|F\|_{W^{2N,2}\left(S\mathbb{X}\right)},
\]
when $m>0$, and
\[
\langle F, \phi_{j,l}^0 \rangle \ll_N \left(|l|^2 +|t_j|^2+1\right)^{-N} \|F\|_{W^{2N,2}\left(S\mathbb{X}\right)}.
\]
Likewise, assuming that the support of $F$ is contained in $y<T$, we have
\[
\left\langle F,E_m\left(\cdot, \frac{1}{2}+it\right)\right\rangle_{S\mathbb{X}}\ll_N \left(|m|^2 +t^2+1\right)^{-N} \|F\|_{W^{2N,2}\left(S\mathbb{X}\right)} \log T,
\]
where we used \cite[(6.1.6)]{kubota} and \cite[(1.6),(1.7)]{dima}.

Now for $m>0$, we take $N=3$ and apply \eqref{per:1} to see that
\[
\sum_{\substack{m> 0\\ 4|m}} \sum_{j=1}^{d_m}\mu_d\left(\phi_{j}^m\right) \sum_{\substack{l \in 4\mathbb{Z}\\ |l|\geq m}}\left\langle F,\phi_{j,l}^m\right\rangle_{S\mathbb{X}}\eta_j^m\left(\phi_{j,l}^m\right) \ll_\epsilon d^{\frac{1}{2}-\frac{25}{512}+\epsilon}  \|F\|_{W^{6,2}\left(S\mathbb{X}\right)},
\]
and for $m=0$, we take $N=2$ and apply \eqref{per:2} to deduce
\[
\sum_{j=1}^{\infty}\mu_d\left(\phi_{j}^0\right) \sum_{l \in 4\mathbb{Z}}\left\langle F,\phi_{j,l}^0\right\rangle_{S\mathbb{X}}\eta_j^0\left(\phi_{j,l}^0\right) \ll_\epsilon d^{\frac{1}{2}-\frac{25}{512}+\epsilon}  \|F\|_{W^{4,2}\left(S\mathbb{X}\right)}.
\]
For the Eisenstein series contribution, we take $N=2$ and apply \eqref{per:3} to see
\[
\sum_{m \in 4\mathbb{Z}}\frac{1}{4\pi}\int_{-\infty}^\infty \left\langle F,E_m\left(\cdot, \frac{1}{2}+it\right)\right\rangle_{S\mathbb{X}} \eta\left(m, \frac{1}{2}+it\right)\mu_d\left(E_0 \left(\cdot, \frac{1}{2}+it\right)\right) dt \ll_\epsilon \log T d^{\frac{7}{16}+\epsilon}\|F\|_{W^{4,2}\left(S\mathbb{X}\right)}.
\]
Therefore Theorem \ref{effect} will follow once we establish the following lower bound for the total length of $\mathscr{C}_d$:
\begin{equation}\label{sie}
l\left(\mathscr{C}_d\right)=2h(d)\log \epsilon_d \gg_\epsilon d^{\frac{1}{2}-\epsilon}.
\end{equation}
To see this, let $d=Db^2$ where $D$ is a fundamental discriminant. Then by Dirichlet class number formula~\cite[p.50]{Davenport} for binary quadratic forms  discriminant $d$  (or  by letting $s\to 1$ in \eqref{classnum}), we have
\[
h(d)\log(\epsilon_d)=d^{\frac{1}{2}}L(1,d)
\]
with the same $L(\cdot,d)$ given in \eqref{l1d}, i.e.,
\[
L(1,d)= L(1,\chi_D)\left(\sum_{a|b} \mu(a) \left(\frac{D}{a} \right) a^{-1} \sigma_{-1}\left(\frac{b}{a}\right)  \right).
\]
Note that
\begin{align*}
\sum_{a|b} \mu(a) \left(\frac{D}{a} \right) a^{-1} \sigma_{-1}\left(\frac{b}{a}\right) &=\sum_{ca|b} \mu(a) \left(\frac{D}{a} \right) \frac{c}{b}
\\
&=\frac{1}{b}\sum_{e|b} e \prod_{p|e} \left( 1-\left(\frac{D}{p}\right) p^{-1} \right),
\end{align*}
where $e=ac$, and that
\[
\frac{1}{b}\sum_{e|b} e \prod_{p|e} \left( 1-\left(\frac{D}{p}\right) p^{-1} \right) \gg b^{-\epsilon}.
\]
Now \eqref{sie} follows by using Siegel's lower bound \cite{siegel}
\[
L(1,\chi_D) \gg_{\epsilon} D^{-\epsilon},
\]
and this completes the proof of Theorem \ref{effect}.

\subsection{Proof of Theorem \ref{main1}}
We are now ready to prove Theorem \ref{main1}. Assume that $\beta:[0,l\left(\beta\right)]\to \mathbb{X}$ is a sufficiently short compact geodesic segment in the region determined by $y<T$ such that $\beta\left([-l\left(\beta\right),2l\left(\beta\right)]\right)$ has no self intersection. (We fix $T$ for simplicity, but it is possible to vary $T$ with $d$.) For $\delta=d^{-a}$ with $a>0$ to be chosen later, such that $l\left(\beta\right) \gg \delta$, let
\[
\beta_1:= \{\beta\left(t\right)~:~ t\in [0,l\left(\beta\right)-\delta]\}
\]
and
\[
\beta_2:= \{\beta\left(t\right)~:~ t\in [-\delta,l\left(\beta\right)]\}.
\]
Then from Lemma \ref{counting}, we have
\[
\frac{1}{\delta^2}\int_{\widetilde{\alpha}_2}\int_{\widetilde{\beta_1}} K_{\delta}^{\theta_1,\theta_2} \left(s_1,s_2\right) ds_1 ds_2 \leq I^{\theta_1,\theta_2} \left(\beta,\alpha_2\right) \leq \frac{1}{\delta^2}\int_{\widetilde{\alpha}_2}\int_{\widetilde{\beta_1}} K_{\delta}^{\theta_1,\theta_2} \left(s_1,s_2\right) ds_1 ds_2
\]
for any closed geodesic $\alpha_2$. Now define $f_1,f_2 \in C_0^\infty \left(S\mathbb{X}\right)$ using Lemma \ref{lem1} by
\[
f_1\left(g\right) = \frac{1}{\delta^2}\int_{\widetilde{\beta_1}} m_{\delta}^{\theta_1,\theta_2} \left(s_1^{-1}g\right) ds_1
\]
and
\[
f_2\left(g\right) =\frac{1}{\delta^2}\int_{\widetilde{\beta_2}} M_{\delta}^{\theta_1,\theta_2} \left(s_1^{-1}g\right) ds_1,
\]
with $\varepsilon = d^{-2a}$, where we assume that $\theta_2-\theta_1 \gg d^{-a}$. Note that $m\left(g_1^{-1}g_2\right)$ and $M\left(g_1^{-1}g_2\right)$ are minorant and majorant of $K_{\delta}^{\theta_1,\theta_2}\left(g_1,g_2\right)$ for $g_1\in \beta_i$, $g_2 \in S\mathbb{X}$ for all sufficiently large $d$. Hence, for  all sufficiently large $d$ (independent of $\alpha_2$), we have
\begin{equation*}
\int_{\widetilde{\alpha}_2} f_1\left(s\right) ds \leq  I^{\theta_1,\theta_2} \left(\beta,\alpha_2\right) \leq \int_{\widetilde{\alpha}_2} f_2\left(s\right) ds,
\end{equation*}
and so
\begin{equation}\label{finales}
\int_{\mathscr{C}_d} f_1\left(s\right) ds \leq  2I^{\theta_1,\theta_2} \left(\beta,C_d\right) \leq \int_{\mathscr{C}_d} f_2\left(s\right) ds,
\end{equation}
where the factor $2$ amounts to the fact that $\mathscr{C}_d$ is a double cover of $C_d$.

We now apply Theorem \ref{effect} to see that
\[
\frac{1}{l\left(\mathscr{C}_d\right)}\int_{\mathscr{C}_d} f_i\left(s\right) ds = \frac{3}{\pi^2} \int_{S\mathbb{X}} f_i\left(g\right) d\mu_g + O_\epsilon\left(d^{-\frac{25}{512}+\epsilon} \|f_i\|_{W^{6,\infty}}\right).
\]
Because of the choice of $f_1$ and $f_2$, we have
\[
 \|f_i\|_{W^{6,\infty}} \ll \varepsilon^{-6} l\left(\beta\right) \ll d^{12a} l\left(\beta\right),
\]
and
\[
\int_{S\mathbb{X}} f_i\left(g\right) d\mu_g = \left(\cos\theta_1-\cos\theta_2\right)  \left(l\left(\beta\right)+O\left(\delta\right)\right)\left(1+O\left(\varepsilon\right)\right) = \left(\cos\theta_1-\cos\theta_2\right) l\left(\beta\right) \left(1+O\left(d^{-2a}\right)\right)
\]
by Lemma \ref{lem1}. Now we complete the proof of Theorem \ref{main1} for sufficiently short geodesic segments by choosing $a = \frac{25}{7168}$ and applying these estimates to \eqref{finales}. This then implies Theorem \ref{main1} for any geodesic segment of length $<1$ by dividing the segment into finitely many sufficiently short geodesic segments, and then applying Theorem \ref{main1} to each of them.

\section{Selberg's pre-Trace Formula for \texorpdfstring{$PSL_2\left(\mathbb{R}\right)$}{PSL2(R)}}\label{pretrace}
Let  $k \in C_0^\infty\left(PSL_2\left(\mathbb{R}\right)\right)$, and let $K$ be the integral kernel on $S\mathbb{X}$ defined by
\[
K\left(g_1,g_2\right) = \sum_{\gamma \in \Gamma} k\left(g_1,\gamma g_2\right),
\]
where $k\left(g_1,g_2\right)=k\left(g_1^{-1} g_2\right)$. The corresponding integral operator $T_K$ acts on $f\in L^2\left(S\mathbb{X}\right)$  by
\[
T_K\left(f\right):=\int_{S\mathbb{X}} K\left(g_1,g_2\right) f\left(g_2\right) dg_2=\int_{PSL_2\left(\RR\right)} k\left(g_1^{-1}g_2\right) f\left(g_2\right) dg_2.
\]
It follows that $T_K(f)\in L^2\left(S\mathbb{X}\right)$.  In this section, we study the spectral expansion of $K$ in terms of the equivariant eigenfunctions of the Casimir operator, which are explicitly described in \S\ref{S1}. In other words, we derive Selberg's pre-trace formula for $PSL_2\left(\mathbb{Z}\right) \backslash PSL_2\left(\mathbb{R}\right)$.

\subsection{Cuspidal spectrum}
In this section, we describe explicitly the spectrum of $T_K$ acting on the cuspidal subspace $L_{\text{cusp}}^2 \left(S\mathbb{X}\right)$.
Let $R_g\left(f\right)\left(x\right)=f\left(xg\right)$ be the right regular action of $PSL_2\left(\mathbb{R}\right)$ on
\[
L_{\text{cusp}}^2 \left(\Gamma\backslash PSL_2\left(\mathbb{R}\right)\right)=L_{\text{cusp}}^2 \left(S\mathbb{X}\right).
\]
\begin{lemma}
Let $\pi$ be an irreducible unitary representation of $PSL_2\left(\mathbb{R}\right)$. Then for any $f\in W_{\pi} \subset L_{\text{cusp}}^2 \left(S\mathbb{X}\right)$, we have
\[
T_K (f) \in  W_{\pi}.
\]
\end{lemma}
\begin{proof}
Observe that
\begin{equation}\label{mot}
T_K(f)\left(g_1\right)=\int_{PSL_2\left(\mathbb{R}\right)} k\left(g_1,g_2\right) f\left(g_2\right) dg_2=\int_{PSL_2\left(\mathbb{R}\right)} k\left(g_1^{-1} g_2\right) f\left(g_2\right) dg_2=\int_{PSL_2\left(\mathbb{R}\right)} k\left(u\right) f\left(g_1u\right) du,
\end{equation}
where $u=g_1^{-1} g_2$. Hence, we have
\[
T_K(f)=\int_{PSL_2\left(\mathbb{R}\right)} k\left(u\right) R_{u}\left(f\right) du,
\]
and because $R_{u}\left(f\right) \in W_{\pi}$ for every $u$, we conclude that $T_K(f)\in W_{\pi}$.
\end{proof}
From \eqref{mot}, for an abstract irreducible unitary representation $\pi$ of $PSL_2\left(\RR\right)$ and $f\in W_{\pi}$, we define the action of $k$ on $f$ by
\[
k*f=\int_{PSL_2\left(\mathbb{R}\right)} k\left(u\right) \pi\left(u\right)\left(f\right) du,
\]
which agrees with $T_K(f)$ when $W_\pi$ is a subspace of $L_{\text{cusp}}^2 \left(S\mathbb{X}\right)$.

Let $\psi: W_{\pi}\to W_{\pi'} $ be an isomorphism of representations $\pi$ and $\pi'$. Note that for $f \in W_\pi$ and $f'\in W_{\pi'}$ with $\psi\left(f\right)=f'$, we have $\psi\left(k*f\right)=k*f'$. We denote by  $\phi_m\in W_{\pi}$  the unique (up to a unit scalar) vector of norm $1$ and weight $m$. We fix the unit scalar except for the spherical or the lowest weight vector, by using the normalized lowering and raising operator that we introduced in~\eqref{phi1} and \eqref{phi2}.

Now let
\begin{equation}\label{defh}
h\left(k,m,n,\pi\right):= \left\langle k*\phi_m,\phi_n \right\rangle,
\end{equation}
and let  $M_{\pi}\left(m,n\right)\left(g\right)=\left\langle\pi\left(g\right)\phi_{m},\phi_{n }\right\rangle$ be the matrix coefficient of $\pi$.  We note that $h\left(k,m,n,\pi\right)$ and  $M_{\pi}\left(m,n\right)\left(g\right)$ do not depend on the choice of the unit scalar of the spherical or the lowest weight vector.

We recall some properties of  $M_{\pi}\left(m,n\right)\left(g\right)$ in the following lemma.
\begin{lemma}\label{triv}
We have for every $g\in PSL_2\left(\mathbb{R}\right)$,
\[
|M_{\pi}\left(m,n\right)\left(g\right)|\leq 1,
\]
and
\[
 M_{\pi}\left(m,n\right)\left(R_{\theta'}  g R_{\theta}\right)=e^{-im\theta}e^{-in\theta'}M_{\pi}\left(m,n\right)\left(g\right).
 \]
 \end{lemma}
\begin{proof}
We have
\[
1=|\pi\left(g\right)\phi_{m}|^2=\sum_{n}\left\langle\pi\left(g\right)\phi_{m},\phi_{n}\right\rangle^2,
\]
from which it is immediate that $ |M_{\pi}\left(m,n\right)\left(g\right)|\leq 1$. For the second identity, we have
\[
 M_{\pi}\left(m,n\right)\left(R_{\theta'}  g R_{\theta}\right)= \left\langle\pi \left(g\right) \pi\left(R_{\theta}\right)\phi_{m},\pi\left(R_{-\theta'}\right)\phi_{n }\right\rangle=e^{-im\theta}e^{-in\theta'}M_{\pi}\left(m,n\right)\left(g\right).\qedhere
\]
\end{proof}
Define $k_{m,n} \in C_0^\infty\left(PSL_2\left(\mathbb{R}\right)\right)$ by
\begin{equation}\label{kmndef}
k_{m,n} \left(g\right) := \frac{1}{4\pi^2}\int_{0}^{2\pi}\int_{0}^{2\pi} k\left(R_{\theta'} g R_{\theta}\right) e^{-in\theta'-im\theta}d\theta' d\theta.
\end{equation}
Note that
\begin{equation}\label{transformation}
k_{m,n} \left(R_{\theta_1}gR_{\theta_2}\right) =  e^{in\theta_1} k_{m,n} \left(g\right) e^{im\theta_2}.
\end{equation}
The following lemma holds for every unitary irreducible   representation  of $PSL_2\left(\RR\right)$.
\begin{lemma}\label{mxc}
We have
\[
h\left(k,m,n,\pi\right)= \int_{PSL_2\left(\mathbb{R}\right)} k_{m,n}\left(u\right) M_{\pi}\left(m,n\right)\left(u\right)du,
\]
and for all non-negative integers $N_1,N_2$, we have the following estimate
\[
h\left(k,m,n,\pi\right) \ll_{N=N_1+N_2} \left(1+|m|\right)^{-N_1} \left(1+|n|\right)^{-N_2}  \|k\|_{W^{N,1}}.
\]
\end{lemma}
\begin{proof}
Recall from the definition that
\[
h\left(k,m,n,\pi\right)= \int_{PSL_2\left(\mathbb{R}\right)} k\left(u\right)\left\langle\pi\left(u\right) \phi_{m},\phi_n\right\rangle du= \int_{PSL_2\left(\mathbb{R}\right)} k\left(u\right) M_{\pi}\left(m,n\right)\left(u\right)du,
\]
and so
\begin{align*}
h\left(k,m,n,\pi\right)&= \int_{PSL_2\left(\mathbb{R}\right)} k\left(u\right) M_{\pi}\left(m,n\right)\left(u\right)du.
\\
&=  \frac{1}{4\pi^2}\int_{PSL_2\left(\mathbb{R}\right)} \int_{\theta}\int_{\theta'} k\left(R_{\theta'}  u R_{\theta}\right)  M_{\pi}\left(m,n\right)\left(R_{\theta'}  u R_{\theta}\right) d\theta d\theta' du
\\
&=  \frac{1}{4\pi^2}\int_{PSL_2\left(\mathbb{R}\right)} M_{\pi}\left(m,n\right)\left( u \right) \int_{\theta}\int_{\theta'} k\left(R_{\theta'}  u R_{\theta}\right) e^{-im\theta}e^{-in\theta'}    d\theta d\theta' du
\\
&= \int_{PSL_2\left(\mathbb{R}\right)} k_{m,n}\left(u\right) M_{\pi}\left(m,n\right)\left(u\right)du.
\end{align*}
Therefore, by integration by parts, we have
\begin{align*}
h\left(k,m,n,\pi\right) \leq \int_{PSL_2\left(\mathbb{R}\right)} |k_{m,n}\left(u\right)| du&= \int_{PSL_2\left(\mathbb{R}\right)} \left| \frac{1}{4\pi^2} \int_{\theta}\int_{\theta'} k\left(R_{\theta'}  u R_{\theta}\right) e^{-im\theta}e^{-in\theta'}    d\theta d\theta' \right|  du\\
&\ll_{N}  \left(1+|m|\right)^{-N_1} \left(1+|n|\right)^{-N_2}  \|k\|_{W^{N,1}},
\end{align*}
where we used  $|M_{\pi}\left(m,n\right)\left( u \right)|\leq 1$ from Lemma \ref{triv}. This completes the proof of our lemma.
\end{proof}

\subsubsection{Principal series representation of \texorpdfstring{$SL_2\left(\RR\right)$}{SL2\left(R\right)}}
For our application in the subsequent chapters, we need a refined estimate for $h\left(k,m,n,\pi\right)$ when $\pi$ is a unitary principal series representation. We first give an explicit representation of $h\left(k,m,n,\pi\right)$.
\begin{lemma}
Let $W_{\pi}$ be a unitary principal series representation of $SL_2\left(\RR\right)$ with the parameter $it$ \cite[Cahpter VII]{knapp}.   Let
\begin{equation}\label{defhr}
h\left(k,m,n,t\right):=\int_{PSL_2\left(\mathbb{R}\right)} k_{m,n}\left(g\right)y^{\frac{1}{2}+it} e^{-im\theta} dg,
\end{equation}
where  $g=na\left(y\right)R_{\theta}$. Then we have
\[
h\left(k,m,n,\pi\right)= h\left(k,m,n,t\right).
\]
\end{lemma}
\begin{proof}
We note that principal series representations are induced from the unitary characters of the upper triangular matrices to $PSL_2\left(\mathbb{R}\right)$ \cite[Cahpter VII]{knapp}. In this model, a dense subspace of a representation is given by
\[
\{f:PSL_2\left(\mathbb{R}\right) \to \CC \text{ continuous }~ :~ f\left(xan\right)=e^{\left(it+\frac{1}{2}\right)\log\left(a\right)}f\left(x\right)  \}
\]
with the norm
\[
|f|^2=\frac{1}{2\pi}\int_{\theta} |f\left(R_{\theta}\right)|^2d\theta,
\]
and the $PSL_2\left(\mathbb{R}\right)$ action is given by
\[
\pi\left(g\right)f\left(x\right)=f\left(g^{-1}x\right).
\]
The weight-$m$ unit vectors are explicitly given by
\[
\phi_m\left(R_\theta a\left(y\right) n\right)=e^{i m\theta}y^{-\left(\frac{1}{2}+it\right)}.
\]
Note that the orthonormal basis $\{\phi_m\}$ is normalized as our convention in \eqref{phi1}, i.e.,
\begin{align*}
\mathbf{E}^- \phi_{m} &= \left(m+1-2it \right)\phi_{m-2}, \text{ and}\\
\mathbf{E}^+ \phi_{m} &= \left(m+1+2it \right)\phi_{m+2}.
\end{align*}
With these, we first see that
\begin{align*}
k*\phi_m\left(R_{\theta'}\right)&= \int_{PSL_2\left(\mathbb{R}\right)} k\left(u\right)y\left(u^{-1}R_{\theta'}\right)^{-\left(\frac{1}{2}+it\right)}e^{im\theta\left(u^{-1}R_{\theta'}\right)} du
\\
&= \int_{PSL_2\left(\mathbb{R}\right)} k\left(R_{\theta'}v^{-1}\right)y\left(v\right)^{-\left(\frac{1}{2}+it\right)}e^{im\theta\left(v\right)} dv,
\end{align*}
where $v= u^{-1}R_{\theta'}$ and $v=R_{\theta\left(v\right)}a\left(y\left(v\right)\right)n\left(v\right)$.
We therefore have
\begin{align*}
h\left(k,m,n,\pi\right)=\left\langle k*f_m,f_n \right\rangle &=\frac{1}{2\pi}  \int_{\theta'} k*f_m\left(R_{\theta'}\right)\bar{f}_n\left(R_{\theta'}\right) d\theta'
\\
&=\frac{1}{2\pi} \int_{\theta'}  e^{-in\theta'}   \int_{PSL_2\left(\mathbb{R}\right)} k\left(R_{\theta'}v^{-1}\right)y\left(v\right)^{-\left(\frac{1}{2}+it\right)}e^{im\theta\left(v\right)} dv d\theta'
\\
&=\frac{1}{2\pi} \int_{PSL_2\left(\mathbb{R}\right)}  y^{\frac{1}{2}+it}  \int_{\theta'}  e^{-in\theta'} e^{-im\theta}  k\left(R_{\theta'}w\right)  d\theta'
dw
\\
&=\int_{PSL_2\left(\mathbb{R}\right)} k_{m,n}\left(w\right)y^{\frac{1}{2}+it} e^{-im\theta} dw,
\end{align*}
where $w=v^{-1}$ and $w=na\left(y\right)R_{\theta}$. Note that $y=y\left(v\right)^{-1}$ and $\theta=-\theta\left(v\right)$.
\end{proof}
We now prove that $h\left(k,m,n,t\right)$ decays fast in all parameters uniformly.
\begin{lemma}\label{important}
Suppose that $k$ is supported inside the compact subset $C\subset SL_2\left(\RR\right)$. Then we have
\[
\int_{PSL_2\left(\mathbb{R}\right)} k_{m,n}\left(g\right)y^{\frac{1}{2}+it} e^{-im\theta} dg \ll_{N,C} \left(1+|m|\right)^{-N_1} \left(1+|n|\right)^{-N_2} \left(1+|t|\right)^{-N_3} \|k\|_{W^{N,\infty}}
\]
for any $N_1,N_2,N_3 \geq 0$, where $N=N_1+N_2+N_3$.
\end{lemma}
\begin{proof}
From the definition, we have
\[
\int_{PSL_2\left(\mathbb{R}\right)} k_{m,n}\left(g\right)y^{\frac{1}{2}+it} e^{-im\theta}dg = \frac{1}{4\pi} \int_{\mathbb{H}}\int_{0}^{2\pi}\int_{0}^{2\pi} k\left(R_{\theta_1'} n\left(x\right)a\left(y\right) R_{\theta_2'}\right)y^{\frac{1}{2}+it} e^{-in\theta_1'-im\theta_2'}d\theta_1' d\theta_2' \frac{dxdy}{y^2},
\]
and so the statement follows from integration by parts.
\end{proof}

\subsection{Continuous spectrum}
For $k_{m,n}$ given by \eqref{kmndef}, let
\begin{equation}\label{kmndef2}
K_{m,n}\left(g_1,g_2\right) := \sum_{\gamma \in \Gamma} k_{m,n}\left(g_1^{-1}\gamma g_2\right).
\end{equation}
Then we infer from \eqref{transformation} that
\[
K_{m,n}\left(g_1 R_{\theta_1}, g_2R_{\theta_2}\right) = e^{-in\theta_1}K_{m,n}\left(g_1,g_2\right)e^{im\theta_2},
\]
and so it defines an integral operator that maps weight $m$-forms to weight $n$-forms. Denote by $S^m \subset L^2 \left(\Gamma \backslash PSL_2\left(\mathbb{R}\right)\right)$ the space of weight $m$ forms and by $S_{\text{cusp}}^m$ the space of weight $m$ forms in $L_{\text{cusp}}^2\left(S\mathbb{X}\right)$. We first recall the following result regarding the decomposition of $K_{m,m}$.
\begin{theorem}[\cite{hejhal}]\label{sel:thm}
The integral kernel
\[
K_{m,m}\left(g_1,g_2\right)- \frac{1}{4\pi}\int_{-\infty}^\infty h\left(k,m,m,t\right) E_m \left(g_1,\frac{1}{2}+it\right)\overline{E_m \left(g_2,\frac{1}{2}+it\right)} dt
\]
defines a compact operator $ S_{\text{cusp}}^m \to S_{\text{cusp}}^m$ that acts trivially on $\Theta$. (Here $h\left(k,m,m,t\right)$ is given by \eqref{defhr}.)
\end{theorem}
We define $\mathbf{E}^{a}$ to be $\left(\mathbf{E}^+\right)^{ a}$ if $a>0$, and $\left(\mathbf{E}^-\right)^{|a|}$ if $a<0$. We have
\[
\overline{\mathbf{E}^{a}}=\left(-\mathbf{E}\right)^{-a},
\]
which follows directly from \eqref{dma}. Let $c_{m,n}$ be given by
\[
\mathbf{E}^{n-m}E_m\left(g,s\right) = c_{m,n}\left(s\right) E_n\left(g,s\right).
\]
Observe that
\[
\mathbf{E}^{n-m} y^s e^{-im\theta} = c_{m,n}\left(s\right) y^s e^{-in\theta},
\]
and that
\begin{equation}\label{symmetry}
\overline{c_{m,n}\left(\frac{1}{2}+it\right)} = c_{n,m}\left(\frac{1}{2}+it\right)
\end{equation}
for $t \in \mathbb{R}$.
\begin{theorem}\label{sel:diff}
For $m,n\in 2\mathbb{Z}$,
\[
K_{m,n}\left(g_1,g_2\right)- \frac{1}{4\pi}\int_{-\infty}^\infty h\left(k,m,n,t\right) E_n \left(g_1,\frac{1}{2}+it\right)\overline{E_m \left(g_2,\frac{1}{2}+it\right)} dt
\]
defines a compact operator $S_{\text{cusp}}^m \to S_{\text{cusp}}^n$ that acts trivially on $\Theta$.
\end{theorem}
\begin{proof}
Note that
\[
\int \mathbf{E}^{m-n}_{g_2}\left( K\left(g_1,g_2\right) f \left(g_2\right) \right) dg_2=0
\]
for every $g_1$,   $m\neq n$, and $ f\in C^{\infty}_0 \left(\Gamma \backslash PSL_2\left(\mathbb{R}\right)\right)$.
Hence
\[
T_{K} \mathbf{E}^{m-n}:  C^{\infty}_0 \left(\Gamma \backslash PSL_2\left(\mathbb{R}\right)\right)  \to C^{\infty}_0 \left(\Gamma \backslash PSL_2\left(\mathbb{R}\right)\right)
\]
is an integral operator with the integral kernel
\[
K'\left(g_1,g_2\right)=\sum_{\gamma \in \Gamma} k' \left(g_1^{-1}\gamma g_2\right),
\]
where
\[
k'\left(g\right)=\left(-\mathbf{E}\right)^{m-n}k \left(g\right)=\overline{\mathbf{E}^{n-m}}k\left(g\right) .
\]
Then by Theorem \ref{sel:thm}, we see that
\[
K''\left(g_1,g_2\right) = K'_{n,n}\left(g_1,g_2\right) -\frac{1}{4\pi} \int_{-\infty}^\infty h\left(k',n,n,t\right) E_n \left(g_1,\frac{1}{2}+it\right)\overline{E_n \left(g_2,\frac{1}{2}+it\right)} dt
\]
defines a compact operator $T_{K''}:S_{\text{cusp}}^n \to S_{\text{cusp}}^n$ that acts trivially on $\Theta$. Note that
\begin{multline*}
\int_{-\infty}^\infty h\left(k',n,n,t\right) E_n \left(g_1,\frac{1}{2}+it\right)\overline{E_n \left(g_2,\frac{1}{2}+it\right)} dt\\
 = \int_{-\infty}^\infty \frac{h\left(k',n,n,t\right)}{\overline{c_{m,n}\left(\frac{1}{2}+it\right)}} E_n \left(g_1,\frac{1}{2}+it\right)\overline{\mathbf{E}^{n-m}E_m \left(g_2,\frac{1}{2}+it\right)} dt.
\end{multline*}
Let
\[
K'''\left(g_1,g_2\right) := K_{m,n}\left(g_1,g_2\right) - \frac{1}{4\pi}\int_{-\infty}^\infty \frac{h\left(k',n,n,t\right)}{\overline{c_{m,n}\left(\frac{1}{2}+it\right)}} E_n \left(g_1,\frac{1}{2}+it\right)\overline{E_m \left(g_2,\frac{1}{2}+it\right)} dt.
\]
Note that
\[
T_{K''} = T_{K'''} \circ \mathbf{E}^{m-n}.
\]
Firstly, since $\mathbf{E}^{m-n}$ does not annihilate the Eisenstein series, $T_{K'''}$ acts trivially on $\Theta$.

If $m>n\geq 0$ or $m<n\leq 0$, then as a map $S_{cusp}^n \to S_{cusp}^m$, $ker\left( \mathbf{E}^{m-n}\right)$ is empty, and we may decompose $S_{cusp}^m$ as
\[
S_{cusp}^m = Im\left(\mathbf{E}^{m-n}\right) \oplus R,
\]
where $R$ is a finite dimensional subspace of $S_{cusp}^m$ spanned by modular forms of weight $>n$ and their images under raising operators in $S_{cusp}^m$. Note that
\[
\left(\mathbf{E}^{m-n}\right)^{-1} :Im\left(\mathbf{E}^{m-n}\right) \to S_{cusp}^n
\]
is a bounded operator, hence
\[
T_{K'''}|_{Im\left(\mathbf{E}^{m-n}\right)} = T_{K''}\circ \left(\mathbf{E}^{m-n}\right)^{-1}
\]
is a compact operator. This implies that $T_{K'''}$ is a direct sum of a compact operator and finite dimensional linear operator, which is a compact operator.

If $n>m\geq 0$ or $n<m\leq 0$, then $\mathbf{E}^{m-n}:S_{\text{cusp}}^n \to S_{\text{cusp}}^m$ is surjective, and so we may define a bounded operator
\[
\left(\mathbf{E}^{m-n}\right)^{-1}: S_m \to \left(ker \left( \mathbf{E}^{m-n}\right)\right)^{\perp}
\]
from which it follows that
\[
T_{K'''} = T_{K''}\circ \left(\mathbf{E}^{m-n}\right)^{-1}
\]
is a compact operator.

If $n>0>m$ or $m>0>n$, then we further decompose $T_{K''}$ to
\[
S_{\text{cusp}}^n \xrightarrow{\mathbf{E}^{-n}} S_{\text{cusp}}^0 \xrightarrow{\mathbf{E}^{m}} S_{\text{cusp}}^m \xrightarrow{T_{K'''}} S_{\text{cusp}}^n,
\]
and then combine the above arguments to see that $T_{K''}$ is a compact operator.

Finally, observe that
\begin{align*}
h\left(k',n,n,t\right) &= \int_{PSL_2\left(\mathbb{R}\right)}\left(\overline{\mathbf{E}^{n-m}} k\left(g\right) \right)y^{\frac{1}{2}+it}e^{in\theta} dg\\
&=c_{n,m}\left(\frac{1}{2}+it\right)\int_{PSL_2\left(\mathbb{R}\right)} k\left(g\right)y^{\frac{1}{2}+it}e^{im\theta} dg,
\end{align*}
and we complete the proof using \eqref{symmetry}.
\end{proof}
\subsection{General case}
We are now ready to describe Selberg's pre-trace formula for $PSL_2\left(\mathbb{R}\right)$.
\begin{theorem}\label{full}
For $k \in C_0^\infty\left(PSL_2\left(\mathbb{R}\right)\right)$, let  $K$ be the integral kernel on $S\mathbb{X}$ defined by
\[
K\left(g_1,g_2\right) = \sum_{\gamma \in \Gamma} k\left(g_1,\gamma g_2\right).
\]
Then we have
\begin{multline*}
K\left(g_1,g_2\right) = \frac{9}{\pi^4}\iint K\left(g_1,g_2\right) dg_1dg_2 +   \sum_{\substack{e\geq 0 \\ 2|e}} \sum_{j=1}^{d_e} \sum_{\substack{m,n \in 2\mathbb{Z}\\ |m|,|n| \geq e}} h\left(k,m,n,\pi_j^{e}\right) \phi_{j,n}^e\left(g_1\right) \overline{\phi_{j,m}^e\left(g_2\right)}\\
+\frac{1}{4\pi}\sum_{m,n\in 2\mathbb{Z}}\int_{-\infty}^\infty h\left(k,m,n,t\right) E_n \left(g_1,\frac{1}{2}+it\right)\overline{E_m \left(g_2,\frac{1}{2}+it\right)} dt,
\end{multline*}
where $\pi_j^{e}$ is the irreducible unitary representation of $PSL_2\left(\mathbb{R}\right)$ associated to $\phi_{j}^e$.
\end{theorem}
\begin{proof}
We first note from \eqref{kmndef} and \eqref{kmndef2} that
\begin{align*}
K_{m,n}\left(g_1,g_2\right) &=\sum_{\gamma \in \Gamma} \frac{1}{4\pi^2}\int_{0}^{2\pi}\int_{0}^{2\pi} k\left(R_{\theta_1'} g_1^{-1} \gamma g_2 R_{\theta_2'}\right) e^{-in\theta_1'-im\theta_2'}d\theta_1' d\theta_2'\\
&=  \frac{1}{4\pi^2}\int_{0}^{2\pi}\int_{0}^{2\pi} \sum_{\gamma \in \Gamma} k\left(R_{-\theta_1'} g_1^{-1} \gamma g_2 R_{\theta_2'}\right) e^{in\theta_1'-im\theta_2'}d\theta_1' d\theta_2'\\
&= \frac{1}{4\pi^2}\int_{0}^{2\pi}\int_{0}^{2\pi} K\left(g_1 R_{\theta_1'},  g_2 R_{\theta_2'}\right) e^{in\theta_1'-im\theta_2'}d\theta_1' d\theta_2'\\
&= \frac{1}{4\pi^2}\int_{0}^{2\pi}\int_{0}^{2\pi} K\left(\left(x_1,y_1,\theta_1'\right), \left(x_2,y_2,\theta_2'\right)\right) e^{in\theta_1'-im\theta_2'}d\theta_1' d\theta_2' e^{-in\theta_1 + im\theta_2}.
\end{align*}
Therefore, we have the Fourier expansion of $K$,
\begin{equation*}
K\left(g_1,g_2\right) = \sum_{n,m \in 2\mathbb{Z}}K_{m,n}\left(g_1,g_2\right),
\end{equation*}
where the summation is uniform for $g_1$ and $g_2$ in compacta.

We infer from Theorem \ref{sel:diff} that
\[
K_{m,n}\left(g_1,g_2\right) -  \frac{1}{4\pi}\int_{-\infty}^\infty h\left(k,m,n,t\right) E_n \left(g_1,\frac{1}{2}+it\right)\overline{E_m \left(g_2,\frac{1}{2}+it\right)} dt
\]
defines a compact operator acting on $L_{\text{cusp}}$ that acts trivially on $\Theta$. Because it only acts non-trivially on weight $m$ forms, we see that
\begin{multline*}
K_{m,n}\left(g_1,g_2\right) -  \frac{1}{4\pi}\int_{-\infty}^\infty h\left(k,m,n,t\right) E_n \left(g_1,\frac{1}{2}+it\right)\overline{E_m \left(g_2,\frac{1}{2}+it\right)} dt \\
= \frac{9}{\pi^4}\iint K_{m,n}\left(g_1,g_2\right) dg_1dg_2 +   \sum_{\substack{e\geq 0 \\ 2|e}}^{\min\{|m|,|n|\}} \sum_{j=1}^{d_e} h\left(k,m,n,\pi_j^e\right) \phi_{j,n}^e\left(g_1\right) \overline{\phi_{j,m}^e\left(g_2\right)},
\end{multline*}
where we used  \eqref{defh}, and the fact that
\[
\int_{-\infty}^\infty h\left(k,m,n,t\right) E_n \left(g_1,\frac{1}{2}+it\right)\overline{E_m \left(g_2,\frac{1}{2}+it\right)} dt
\]
acts trivially on $L_{\text{cusp}}^2$. Note that the integral on the right-hand side of the equation vanishes unless $m=n=0$, in which case it is identical to
\[
\frac{9}{\pi^4}\iint K\left(g_1,g_2\right) dg_1dg_2.\qedhere
\]
\end{proof}

\subsection{Proof of Theorem \ref{main}}
We now present a proof of Theorem \ref{main}. Recall from Theorem \ref{full} that we have
\begin{multline*}
K\left(g_1,g_2\right) = \frac{9}{\pi^4}\iint K\left(g_1,g_2\right) dg_1dg_2 +   \sum_{\substack{e\geq 0 \\ 2|e}} \sum_{j=1}^{d_e} \sum_{\substack{m,n \in 2\mathbb{Z}\\ |m|,|n| \geq e}} h\left(k,m,n,\pi_j^{e}\right) \phi_{j,n}^e\left(g_1\right) \overline{\phi_{j,m}^e\left(g_2\right)}\\
+\frac{1}{4\pi}\sum_{m,n\in 2\mathbb{Z}}\int_{-\infty}^\infty h\left(k,m,n,t\right) E_n \left(g_1,\frac{1}{2}+it\right)\overline{E_m \left(g_2,\frac{1}{2}+it\right)} dt.
\end{multline*}
Therefore we have
\[
\frac{1}{l\left(\mathscr{C}_{d_1}\right)l\left(\mathscr{C}_{d_2}\right)}\int_{\mathscr{C}_{d_2}}\int_{\mathscr{C}_{d_1}}K\left(s_1,s_2\right) ds_1ds_2 = M+D+\frac{1}{4\pi}E,
\]
where
\[
M = \frac{9}{\pi^4}\iint K\left(g_1,g_2\right) dg_1dg_2,
\]
\begin{multline*}
D= \sum_{\substack{e\geq 0 \\ 2|e}} \sum_{j=1}^{d_e} \sum_{\substack{m,n \in 2\mathbb{Z}\\ |m|,|n| \geq e}} h\left(k,m,n,\pi_j^{e}\right) \frac{\mu_{d_1}\left(\phi_{j,n}^e\right)}{l\left(\mathscr{C}_{d_1}\right)} \frac{\overline{\mu_{d_2}\left(\phi_{j,m}^e\right)}}{l\left(\mathscr{C}_{d_2}\right)}\\
=\sum_{\substack{e\geq 0 \\ 4|e}} \sum_{j=1}^{d_e} \frac{\mu_{d_1}\left(\phi_{j}^e\right)}{l\left(\mathscr{C}_{d_1}\right)}\frac{\overline{\mu_{d_2}\left(\phi_{j}^e\right)}}{l\left(\mathscr{C}_{d_2}\right)}\sum_{\substack{m,n \in 4\mathbb{Z}\\ |m|,|n| \geq e}}  h\left(k,m,n,\pi_j^{e}\right) \eta_j^e\left(\phi_{j,n}^e\right) \overline{\eta_j^e\left(\phi_{j,m}^e\right)},
\end{multline*}
and
\begin{multline*}
E=\sum_{m,n\in 2\mathbb{Z}}\int_{-\infty}^\infty h\left(k,m,n,t\right) \frac{\mu_{d_1}\left(E_n \left(\cdot,\frac{1}{2}+it\right)\right)}{l\left(\mathscr{C}_{d_1}\right)}\frac{\overline{\mu_{d_2}\left(E_m \left(\cdot,\frac{1}{2}+it\right)\right)}}{l\left(\mathscr{C}_{d_2}\right)} dt\\
=\sum_{m,n\in 4\mathbb{Z}}\int_{-\infty}^\infty h\left(k,m,n,t\right) \frac{\mu_{d_1}\left(E_0 \left(\cdot,\frac{1}{2}+it\right)\right)}{l\left(\mathscr{C}_{d_1}\right)}\frac{\overline{\mu_{d_2}\left(E_0 \left(\cdot,\frac{1}{2}+it\right)\right)}}{l\left(\mathscr{C}_{d_2}\right)} \eta\left(n,\frac{1}{2}+it\right)\overline{\eta\left(m,\frac{1}{2}+it\right)} dt.
\end{multline*}
For $D$ with $e>0$, we use \eqref{period:2}, \eqref{per:1}, Lemma \ref{mxc} with $N_1=N_2=5$, and \eqref{sie} to see that
\begin{align*}
 &\sum_{\substack{e > 0 \\ 4|e}} \sum_{j=1}^{d_e} \frac{\mu_{d_1}\left(\phi_{j}^e\right)}{l\left(\mathscr{C}_{d_1}\right)}\frac{\overline{\mu_{d_2}\left(\phi_{j}^e\right)}}{l\left(\mathscr{C}_{d_2}\right)}\sum_{\substack{m,n \in 4\mathbb{Z}\\ |m|,|n| \geq e}}  h\left(k,m,n,\pi_j^{e}\right)  \eta_j^e\left(\phi_{j,n}^e\right) \overline{\eta_j^e\left(\phi_{j,m}^e\right)}\\
\ll_\epsilon &\sum_{\substack{e > 0 \\ 4|e}} e^{6.8} \left(d_1d_2\right)^{-\frac{25}{512}+\epsilon}\sum_{\substack{m,n \in 4\mathbb{Z}\\ |m|,|n| \geq e}} |m|^{-5}|n|^{-5} \|k\|_{W^{10,\infty}}\\
\ll & \left(d_1d_2\right)^{-\frac{25}{512}+\epsilon} \|k\|_{W^{10,\infty}}.
\end{align*}
For $D$ with $e=0$, we use \eqref{period:1}, \eqref{per:2}, Lemma \ref{important} with $N_1=N_2=2$ and $N_3=4$, and \eqref{sie} to see that
\begin{align*}
&\sum_{j=1}^{\infty} \frac{\mu_{d_1}\left(\phi_{j}^0\right)}{l\left(\mathscr{C}_{d_1}\right)}\frac{\overline{\mu_{d_2}\left(\phi_{j}^0\right)}}{l\left(\mathscr{C}_{d_2}\right)}\sum_{\substack{m,n \in 4\mathbb{Z}}}  h\left(k,m,n,\pi_j^{0}\right)  \eta_j^0\left(\phi_{j,n}^0\right) \overline{\eta_j^0\left(\phi_{j,m}^0\right)}\\
\ll_\epsilon  &\sum_{j=1}^{\infty} \left(d_1d_2\right)^{-\frac{25}{512}+\epsilon} |t_j|^{\frac{3}{2}} \sum_{\substack{m,n \in 4\mathbb{Z}}} \left(1+|m|\right)^{-2}\left(1+|n|\right)^{-2} \left(1+|t_j|\right)^{-4} \|k\|_{W^{8,\infty}}\\
\ll &\left(d_1d_2\right)^{-\frac{25}{512}+\epsilon}\|k\|_{W^{8,\infty}}.
\end{align*}
For $E$, we use \eqref{period:3}, \eqref{per:3}, Lemma \ref{important} with $N_1=N_2=2$ and $N_3=3$, and \eqref{sie} to see that
\begin{align*}
&\sum_{m,n\in 4\mathbb{Z}}\int_{-\infty}^\infty h\left(k,m,n,t\right)  \frac{\mu_{d_1}\left(E_0 \left(\cdot,\frac{1}{2}+it\right)\right)}{l\left(\mathscr{C}_{d_1}\right)}\frac{\overline{\mu_{d_2}\left(E_0 \left(\cdot,\frac{1}{2}+it\right)\right)}}{l\left(\mathscr{C}_{d_2}\right)} \eta\left(n,\frac{1}{2}+it\right)\overline{\eta\left(m,\frac{1}{2}+it\right)} dt\\
\ll_\epsilon &\sum_{m,n\in 4\mathbb{Z}}\int_{-\infty}^\infty   \left(d_1d_2\right)^{-\frac{1}{16}+\epsilon} \left(1+|m|\right)^{-2}\left(1+|n|\right)^{-2}\left(|t|+1\right)^{-2}\|k\|_{W^{7,\infty}} dt\\
\ll  &\left(d_1d_2\right)^{-\frac{1}{16}+\epsilon} \|k\|_{W^{7,\infty}}.
\end{align*}
Now observe that
\[
\iint K\left(g_1,g_2\right) dg_1dg_2 = \int_{S\mathbb{X}} \int_{S\mathbb{H}} k\left(g_1^{-1}g_2\right) dg_2dg_1 = \frac{\pi^2}{3} \int_{S\mathbb{H}} k\left(g\right)dg,
\]
and so
\[
M =  \frac{3}{\pi^2} \int_{S\mathbb{H}} k\left(g\right)dg.
\]
So far, we proved the following:
\begin{theorem}\label{mainlem}
For any $k \in C_0^\infty \left(S\mathbb{H}\right)$, we have
\[
\frac{1}{l\left(\mathscr{C}_{d_1}\right)l\left(\mathscr{C}_{d_2}\right)}\int_{\mathscr{C}_{d_2}}\int_{\mathscr{C}_{d_1}}K\left(s_1,s_2\right) ds_1ds_2 = \frac{3}{\pi^2} \int_{S\mathbb{H}} k\left(g\right)dg + O_\epsilon\left(\left(d_1d_2\right)^{-\frac{25}{512}+\epsilon}\|k\|_{W^{10,\infty}}\right).
\]
\end{theorem}
\begin{remark}
Note that this is \textit{not} the same as equidistribution of $\mathscr{C}_{d_1} \times \mathscr{C}_{d_2}$ in $S\mathbb{X} \times S\mathbb{X}$. For instance, if we replace $K$ with any compactly supported smooth function in $S\mathbb{X} \times S\mathbb{X}$, then the equality may not hold when $d_1$ is fixed and $d_2$ tends to $\infty$.
\end{remark}
In order to prove Theorem \ref{main}, we make specific choices of $k$ in Theorem \ref{mainlem}. We let $K_1$ and $K_2$ to be the kernel corresponding to $k=m_\delta^{\theta_1,\theta_2}$ and $k=M_\delta^{\theta_1,\theta_2}$ defined in Lemma \ref{lem1}, respectively. Then by Lemma \ref{lem1}, we have
\begin{multline*}
\frac{1}{l\left(\mathscr{C}_{d_1}\right)l\left(\mathscr{C}_{d_2}\right)}\int_{\mathscr{C}_{d_2}}\int_{\mathscr{C}_{d_1}}K_1\left(s_1,s_2\right) ds_1ds_2 \leq \frac{1}{l\left(\mathscr{C}_{d_1}\right)l\left(\mathscr{C}_{d_2}\right)}\int_{\mathscr{C}_{d_2}}\int_{\mathscr{C}_{d_1}}K_\delta^{\theta_1,\theta_2}\left(s_1,s_2\right) ds_1ds_2\\
\leq \frac{1}{l\left(\mathscr{C}_{d_1}\right)l\left(\mathscr{C}_{d_2}\right)}\int_{\mathscr{C}_{d_2}}\int_{\mathscr{C}_{d_1}}K_2\left(s_1,s_2\right) ds_1ds_2,
\end{multline*}
while we know from Lemma \ref{lemma1} that
\[
\int_{\mathscr{C}_{d_2}}\int_{\mathscr{C}_{d_1}}K_\delta^{\theta_1,\theta_2}\left(s_1,s_2\right) ds_1ds_2 = 4\delta^2 I_{\theta_1,\theta_2} \left(C_{d_1},C_{d_2}\right).
\]
We now apply Theorem \ref{mainlem} and Lemma \ref{lem1} to see that
\[
\frac{1}{l\left(\mathscr{C}_{d_1}\right)l\left(\mathscr{C}_{d_2}\right)}\int_{\mathscr{C}_{d_2}}\int_{\mathscr{C}_{d_1}}K_i\left(s_1,s_2\right) ds_1ds_2 = \frac{3}{\pi^2}\left(\cos\theta_1-\cos\theta_2\right) \delta^2  \left(1+O\left(\varepsilon\right)\right) +O_\epsilon\left(\left(d_1d_2\right)^{-\frac{25}{512}+\epsilon}\varepsilon^{-10}\right).
\]
Therefore, we have
\[
\frac{I_{\theta_1,\theta_2} \left(C_{d_1},C_{d_2}\right)}{l\left(C_{d_1}\right)l\left(C_{d_2}\right)} = \frac{3}{\pi^2} \left(\cos\theta_1-\cos\theta_2\right) \left(1+O\left(\delta^2\right)\right)\left(1+O\left(\varepsilon\right)\right)+O_\epsilon\left(\left(d_1d_2\right)^{-\frac{25}{512}+\epsilon}\varepsilon^{-10}\delta^{-2}\right),
\]
and by choosing $\delta^2 = \varepsilon = \left(d_1d_2\right)^{-\frac{25}{6144}}$, we complete the proof of Theorem \ref{main}.

\bibliography{20210224}
\bibliographystyle{alpha}

\appendix
\section{Jacobian computation}\label{app2}
Recall that $\Psi:AKA \to SL_2\left(\mathbb{R}\right)$ is given by
\[
\left(t_1,\varphi,t_2\right) \mapsto \begin{pmatrix}e^{\frac{t_1}{2}} & 0 \\ 0 & e^{-\frac{t_1}{2}}\end{pmatrix} R_{\frac{\varphi}{2}}\begin{pmatrix}e^{-\frac{t_2}{2}} & 0 \\ 0 & e^{\frac{t_2}{2}}\end{pmatrix}=\begin{pmatrix}e^{\frac{t_1-t_2}{2}}\cos \frac{\varphi}{2} &-e^{\frac{t_1+t_2}{2}}\sin \frac{\varphi}{2} \\e^{\frac{-t_1-t_2}{2}}\sin \frac{\varphi}{2} &e^{\frac{t_2-t_1}{2}}\cos \frac{\varphi}{2} \end{pmatrix}.
\]
In this section, we compute the pullback of $dV = \frac{dxdyd\theta}{y^2}$ under $\Psi$. We start with the identity
\[
\begin{pmatrix}e^{\frac{t_1-t_2}{2}}\cos \frac{\varphi}{2} &-e^{\frac{t_1+t_2}{2}}\sin \frac{\varphi}{2} \\e^{\frac{-t_1-t_2}{2}}\sin \frac{\varphi}{2} &e^{\frac{t_2-t_1}{2}}\cos \frac{\varphi}{2} \end{pmatrix}= n\left(x\right)a\left(y\right)R_\theta = \begin{pmatrix}*&*\\\frac{\sin \theta}{\sqrt{y}}&\frac{\cos \theta}{\sqrt{y}}\end{pmatrix}.
\]
By comparing the image of $i \in \mathbb{H}$, we have
\[
x+iy = \frac{e^{\frac{t_1-t_2}{2}}\cos \frac{\varphi}{2}i -e^{\frac{t_1+t_2}{2}}\sin \frac{\varphi}{2} }{e^{\frac{-t_1-t_2}{2}}\sin \frac{\varphi}{2}i +e^{\frac{t_2-t_1}{2}}\cos \frac{\varphi}{2}},
\]
and for simplicity, we write this as $\frac{A}{B}$. By comparing the second row of each matrix, we have
\[
\frac{e^{i\theta}}{\sqrt{y}} = B.
\]
From a quick computation, we see that
\[
A_{t_1} =\frac{A}{2} ,~ B_{t_1} = -\frac{B}{2},~A_{t_2} = \frac{\overline{A}}{2}, ~B_{t_2} = \frac{\overline{B}}{2},~A_{\varphi} = -\frac{e^{t_1}}{2}B,~ B_\varphi = \frac{e^{-t_1}}{2}A,~\text{Im}A\overline{B} =1,~y= \frac{1}{|B|^2}.
\]
We use these to express the Jacobian matrix in terms of $A$ and $B$ as follows
\[
\frac{\partial\left(x,y,\theta\right)}{\partial\left(t_1,t_2,\varphi\right)} =
\begin{pmatrix}
\text{Re}\frac{A}{B}&\text{Im}\frac{1}{B^2}&\text{Re}\left(-\frac{e^{t_1}}{2}-\frac{e^{-t_1}}{2}\frac{A^2}{B^2}\right)\\
\text{Im}\frac{A}{B}&-\text{Re}\frac{1}{B^2}&\text{Im}\left(-\frac{e^{t_1}}{2}-\frac{e^{-t_1}}{2}\frac{A^2}{B^2}\right)\\
0&\frac{1}{2}\text{Im}\frac{\overline{B}}{B}&\frac{e^{-t_1}}{2|B|^2}
\end{pmatrix}.
\]
From this, we have
\begin{align*}
\frac{1}{y^2}\left|\frac{\partial\left(x,y,\theta\right)}{\partial\left(t_1,t_2,\varphi\right)}\right|
&=|B|^4\left|\frac{\partial\left(x,y,\theta\right)}{\partial\left(t_1,t_2,\varphi\right)}\right| \\
&= \left| -\frac{1}{2}e^{-t_1}\text{Re}\left(\frac{\overline{A}}{B}\right) +\frac{1}{4}\text{Im}\left(\overline{B}^2\right)\text{Im}\left(\overline{A}B\left(e^{t_1}+e^{-t_1}\frac{A^2}{B^2}\right)\right)\right|\\
&=\left|\frac{e^{t_1}}{2}\text{Im}\left(B^2\right) + \frac{e^{-t_1}}{4|B|^2} \left(-2\text{Re}\left(AB\right) - |A|^2\text{Im}\left(B^2\right)\right)\right|.
\end{align*}
Now we use the definition of $A$ and $B$ to compute each term explicitly as follows
\begin{align*}
2\text{Re}\left(AB\right)&= -\left(e^{t_2}+e^{-t_1}\right)\sin\varphi\\
e^{t_1}\text{Im}\left(B^2\right)&=\sin\varphi\\
e^{-t_1}|A|^2&=e^{t_2} \sin^2 \frac{\varphi}{2}+e^{-t_2} \cos^2 \frac{\varphi}{2}\\
e^{t_1}|B|^2&=e^{-t_2} \sin^2 \frac{\varphi}{2}+e^{t_2} \cos^2 \frac{\varphi}{2},
\end{align*}
and so
\[
\frac{1}{y^2}\left|\frac{\partial\left(x,y,\theta\right)}{\partial\left(t_1,t_2,\varphi\right)}\right| = \frac{1}{2}|\sin \varphi|.
\]
Therefore, we conclude that
\begin{equation}\label{app:1}
dV = \frac{1}{2}|\sin \varphi| dt_1 dt_2 d\varphi.
\end{equation}

\end{document}